\documentclass[oneside,a4paper, 12pt]{amsart}
\usepackage{mathptmx,amssymb}
\usepackage[scale={.7,.8}]{geometry}
\usepackage{mathrsfs}
\theoremstyle{plain}
\newtheorem{theorem}{Theorem}
\newtheorem{lemma}{Lemma}
\newtheorem{corollary}{Corollary}
\newtheorem{proposition}{Proposition}

\theoremstyle{definition}
\newtheorem*{definition}{Definition}
\newtheorem{remark}{Remark}
\title{Affine processes on $\mathbb{R}_+^m\times\mathbb{R}^n$ and multiparameter time changes}
\author[Caballero]{Ma. Emilia Caballero}
\address[MEC and GUB]{Instituto de Matem\'aticas, Universidad Nacional Aut\'onoma de M\'exico.  
Distrito Federal CP 04510,  M\'exico}
\author[P\'erez-Garmendia]{Jos\'e Luis P\'erez Garmendia}
\address[JLPG]{
Instituto de Investigaciones en Matem\'aticas Aplicadas y en Sistemas, 
Universidad Nacional Aut\'onoma de M\'exico. 
Distrito Federal CP 04510,  M\'exico}
\author[Uribe Bravo]{Ger\'onimo Uribe Bravo}
\thanks{Research supported by UNAM-DGAPA-PAPIIT grant no. IA101014} 
\email{marie@matem.unam.mx\\ garmendia@sigma.iimas.unam.mx\\ geronimo@matem.unam.mx}
\newenvironment{lesn}{\begin{linenomath}\begin{equation*}}{\end{equation*}\end{linenomath}}
\subjclass[2010]{60J80, 60F17}
\keywords{L\'evy processes, Continuous branching processes with immigration, Ornstein-Uhlenbeck processes, multiparameter time change%
}
\usepackage[dvipsnames]{xcolor}
\usepackage[colorlinks,citecolor=Plum,urlcolor=Periwinkle,linkcolor=DarkOrchid]{hyperref}
\usepackage[mathlines]{lineno}

\newcommand{\st}{\ensuremath{\mathbf{E}}}
\newcommand{\Fc}{\F^\circ}
\newcommand{\re}{\ensuremath{\mathbb{R}}}
\newcommand{\paren}[1]{\ensuremath{\left( #1\right) }}
\newcommand{\F}{\ensuremath{\mathscr{F}}}
\newcommand{\oo}{\ensuremath{ \Omega  } }
\newcommand{\p}{\mathbb{P}}
\newenvironment{esn}{\begin{equation*}}{\end{equation*}}
\newcommand{\imf}[2]{\ensuremath{#1\!\paren{#2}}}
\newcommand{\se}{\ensuremath{\mathbb{E}}}
\newcommand{\cond}[2]{\left.\vphantom{#2}#1\ \right| #2}
\DeclareMathOperator{\cb}{CB}
\DeclareMathOperator{\cbi}{CBI}
\newcommand{\defin}[1]{{\bf #1}}
\newcommand{\bra}[1]{\ensuremath{\left[ #1\right] }}
\newcommand{\z}{\ensuremath{\mathbb{Z}}}
\newcommand{\na}{\ensuremath{\mathbb{N}}}
\newcommand{\set}[1]{\ensuremath{\left\{ #1\right\} }}
\DeclareMathOperator{\id}{Id} %
\newcommand{\indi}[1]{\si_{#1}}
\newcommand{\si}{{\ensuremath{\bf{1}}}}
\newcommand{\eps}{\ensuremath{ \varepsilon}}
\newcommand{\esp}[1]{\ensuremath{\se\! \left( #1 \right)}}
\newcommand{\cadlag}{c\`adl\`ag}
\DeclareMathOperator{\ivp}{IVP} %
\newcommand{\sa}{\ensuremath{\sigma}\nbd field}
\newcommand{\nbd}{\nobreakdash -}
\newcommand{\sag}[1]{\sigma\!\paren{#1}}
\newcommand{\ra}{\ensuremath{\mathbb{Q}}}
\newcommand{\proba}[1]{\ensuremath{\sip\! \left( #1 \right)}}
\newcommand{\sip}{\mathbb{P}}
\newcommand{\mc}[1]{\ensuremath{\mathscr{#1}}}
\newcommand{\G}{\ensuremath{\mc{G}}}
\newcommand{\floor}[1]{\ensuremath{\lfloor #1\rfloor}}
\begin{document}
\begin{abstract}
We present a time change construction of affine processes with state-space $\mathbb{R}_+^m\times \mathbb{R}^n$. 
These processes were systematically studied in \cite{MR1994043} since they contain interesting classes of processes such as L\'evy processes, continuous branching processes with immigration, and of the Ornstein-Uhlenbeck type. 
The construction is based on a (basically) continuous functional of a multidimensional L\'evy process which implies that limit theorems for L\'evy processes 
(both almost sure and in distribution) can be inherited to affine processes. 
The construction can be interpreted as a multiparameter time change scheme or as a (random) ordinary differential equation driven by discontinuous functions. 
In particular, we propose approximation schemes for affine processes based on the Euler method for solving the associated discontinuous ODEs, which are shown to converge. 
\end{abstract}
\maketitle
\section{Introduction}
\label{introSection}

Affine processes on the state-space $\st=\re_+^m\times \re^n$ are a class of processes introduced in \cite{MR1994043} for two reasons. 
First, they contain important classes of Markov processes like
L\'evy processes, (multi-type) continuous branching processes with immigration, and of the Ornstein-Uhlen\-beck type. 
That is, they contain the fundamental examples of models in (stochastic) population dynamics (as in \cite{MR2466449}) and mathematical finance (as has been argued in \cite{MR1994043} and \cite{MR2233549}). 
Second, they are analytically tractable. 
Indeed, they have been shown to be parametrized in a manner similar to L\'evy processes and one can access 
their finite dimensional distributions by solving an ordinary differential equation of the Riccati type (cf. \cite{MR1994043}). 

To define them, let 
$Z=(Z_t,t\geq 0)$ denote a stochastic process on a measurable space $\paren{\oo,\F}$ whose paths are c\'adlag functions from $[0,\infty)$ to $\st$. 
The canonical filtration of $Z$ will be denoted $\Fc_t$. 
Suppose that the measurable space is equipped with a family of (sub)probability measures $\paren{\p_z,z\in \st}$ such that under each $\p_z$ the process $Z$ starts at $z$. 
Furthermore, we assume that $Z$ is stochastically continuous under $\p_z$ for any $z\in\st$ and that these measures constitute a Markov family:
\begin{linenomath}\begin{esn}
\imf{\se_z}{\cond{ \imf{f}{Z_{t+s}}}{\Fc_s}}=\imf{P_tf}{Z_s}\quad\text{where}\quad \imf{P_tf}{z}=\imf{\se_z}{\imf{f}{Z_t}}. 
\end{esn}\end{linenomath}
\begin{definition}
The Markov family $\paren{\p_z,z\in \st}$ is affine if
\begin{linenomath}
\begin{equation}
\label{affineExponentEquation}
\imf{\se_z}{e^{ u\cdot Z_t}}=\imf{\Phi}{t,u}e^{z\cdot \imf{\psi}{t,u}}
\end{equation}\end{linenomath}for all $u\in 
\re_-^m\times i\re^n$, where $\re_-^m$ denotes the set of elements of $\re^m$ whose coordinates are negative. 
\end{definition}
These processes are part of a larger one of so-called affine processes on general state-spaces and much recent work has been aimed at characterizing these Markov processes, for example by proving their Feller property and the precise form of their infinitesimal generator. This work started in \cite{MR1994043} for \emph{regular} affine processes on $\st$ and was later extended in \cite{MR2851694}, \cite{MR3040553} and \cite{MR3185916} by proving that regularity already follows from stochastic continuity and also by considering more general state-spaces than we do here. 

Our main result aims at giving a pathwise construction of affine processes in terms of a multiparameter time change of L\'evy processes, which are considered as more basic building blocks. 
\begin{theorem}
\label{affineProcessRepresentationTheorem}
Let $X^1, \ldots, X^m$ and $Y$ be independent L\'evy processes on $\re^{m+n}$. 
We suppose that the first $m$ coordinates of $Y$ are subordinators, 
that $X^{i,i}$ has no negative jumps and that $X^{i,j}$ is a subordinator for $1\leq i,j\leq m$ and $i\neq j$. 
Furthermore, in the Gaussian part of $X^i$, the $i$-th coordinate is assumed independent of coordinates $m+1$ up to $n$. 

Let $\beta$ be an $n\times n$ real matrix. 
Then, for any $z\in \st$ there exists a unique solution $Z$ to\begin{equation}
\label{affineProcessRepresentationEquation}
\begin{cases}
Z^j_t=z_j+\sum_{i=1}^m X^{i,j}\circ C^i_t+Y^j_t   &1\leq j\leq m \\
Z^{m+j}_t=z_{m+j}+\sum_{i=1}^m X^{i,m+j}\circ C^i_t+Y^{m+j}_t+\sum_{i=1}^{n}C^{m+i}_t\beta_{i,j} & 1\leq j\leq n
\end{cases}
\end{equation}with\begin{linenomath}\begin{esn}
C^i_t=\int_0^t Z^i_s\, ds
\end{esn}\end{linenomath}whose first $m$ coordinates are non-negative. If $\p_z$ denotes the law of $Z$, then $\paren{\p_z,z\in\st}$ is an affine Markov family on $\st$ and every  affine Markov family $\st$ is obtained by this construction. 
\end{theorem}

Note that the non-negative coordinates are more difficult to handle. 
Indeed, the non-negative coordinates alone constitute an affine process with $n=0$, which is then called a multitype continuous-state branching process with immigration ($\cbi$) introduced (without immigration and with $m=2$) in \cite{MR0234531}. 
Once we analyze the case $n=0$, we will then get the general case by solving a linear differential equation driven by the solution when $n=0$. 
These real-valued coordinates constitute the Ornstein-Uhlenbeck part of the process, which is now not only driven by a L\'evy process but also by a sum of time changed L\'evy processes. 

Equation \eqref{affineProcessRepresentationEquation} represents a multiparameter time change equation proposed in  \cite{MR577310} to generalize the classical time change construction of Markov processes of Volkonskii (cf. \cite{MR0100919}, \cite[Vol 1, Ch 10]{MR0193671}). 
A multiparameter time change representation of affine processes was first proposed (in a weak sense) in \cite{MR2233549}; in that paper, the question of whether the affine process was adapted to the filtration of the L\'evy process was left open. 
Recently, there have been a number of results concerning this time change representation. 
For example, the PhD thesis \cite{gabrielliPhDThesis} (the relevant chapter is found in \cite{Gabrielli:2014aa}) 
proves existence 
(under additional but minor technical assumptions) 
for  a time change representation as in Theorem \ref{affineProcessRepresentationTheorem}.
%
A discrete space version of Theorem \ref{affineProcessRepresentationTheorem} has also been recently studied. 
Indeed, a construction of Galton-Watson processes (without immigration) in terms of multiparameter time changes of random walks is found in \cite{Chaumont:2013aa} in discrete time and  \cite{Chaumont:2014aa} in continuous time. 
More generally, the connection between time changes and changes of measure and the application to mathematical finance is explored in \cite{MR2779876}. 
The main contribution of our work is that we prove uniqueness of the pathwise representation in \eqref{affineProcessRepresentationEquation} (as well as for an accompanying inequality). 
Uniqueness is the main tool in the forthcoming stability analysis of the pathwise representation. 

We now state some continuity properties of the system of equations \eqref{affineProcessRepresentationEquation}. 
Consider a sequence, indexed by $l\geq 1$, of $m+1$ stochastic processes $X^{1,\cdot,l},\ldots, X^{m,\cdot,l}, Y^l$ which satisfy the upcoming hypothesis \defin{H} of p. \pageref{hypothesisH}. 
Consider also any sequence of numbers $0\leq \sigma_l\to 0$. 
The number $\sigma_l$ is interpreted as the discretization parameter to be used in an Euler type scheme as follows. 
When $\sigma_l>0$, let $Z^{j,l}$ and $C^{j,l}$, $1\leq j\leq m+n$, be defined recursively by means of
\begin{linenomath}
\begin{align}
\label{approximationAffineProcessRepresentationEquation}
C^{j,l}_0=0,\quad
Z^{j,l}_{\sigma_l k}&=\bra{\sum_{i=1}^m z^j_l+X^{i,j,l}\circ C^{i,l}_{\sigma_l k} +Y^{j,l}_{\sigma_l k}}^+
\quad
\text{and}\quad  C^{j,l}_{\sigma_l \paren{k+1}}=C^{j,l}_{\sigma_l k}+ Z^{j,l}_{\sigma_l k}\sigma_l
\intertext{when $1\leq j\leq m$, while for $1\leq j\leq n$ we only change the definition of $Z^{m+j,l}$ to}
\label{approximationOUAffineProcessRepresentationEquation}
Z^{m+j,l}_{\sigma_l k}&=z^j_l+\sum_{i=1}^m X^{i,m+j,l}\circ C^{i,l}_{\sigma_l k} +Y^{m+j,l}_{\sigma_l k}+\sum_{i=1}^n C^{m+i,l}_{\sigma_l k}\beta_{i,j}
\end{align}
\end{linenomath}
When $\sigma_l=0$, the forthcoming Lemma \ref{generalExistenceLemma} asserts that \eqref{affineProcessRepresentationEquation}, when driven by $X^{i,j,l}, Y^{j,l}$, admits a (global) solution (which could, in principle, explode). In that case, we let $Z^{j,l},C^{j,l}$ be any such solution. 
We recall in Subsection \ref{stabilityAnalysisSubsection} the definitions of the Skorohod $J_1$ topology and of the uniform $J_1$ topology. 
\begin{theorem}
\label{stochasticLimitTheorem}
Let $X^1,\ldots, X^m$ and $Y$ be as in Theorem \ref{affineProcessRepresentationTheorem}. 
Let $Z,C$ be the unique processes satisfying \eqref{affineProcessRepresentationEquation}. 

Suppose that $X^{1,\cdot, l},\ldots, X^{m,\cdot,l}, Y^{\cdot, l}$ are stochastic processes which satisfy hypothesis \defin{H} of p.  \pageref{hypothesisH} and such that 
$X^{i,\cdot, l}$ converges to $X^{i,\cdot }$ (and $Y^{\cdot, l}$ converges to $Y$) as $l\to\infty$. 
(The convergence can be weak or almost surely in the Skorohod $J_1$ topology when \eqref{affineProcessRepresentationEquation} has no explosion and in the uniform $J_1$ topology in case of explosion. ) Assume that $z^j_l\to z^j$. 

If $Z^{\cdot,l}$, $C^{\cdot, l}$ are any processes satisfying 
\eqref{approximationAffineProcessRepresentationEquation} and \eqref{approximationOUAffineProcessRepresentationEquation} when $\sigma_l>0$ or 
\eqref{affineProcessRepresentationEquation} with respect to the driving processes $X^{1,\cdot, l},\ldots, X^{m,\cdot,l}, Y^{\cdot, l}$ when $\sigma_l=0$
then $C^l\to C$ (with respect to the topology of uniform convergence on compact sets when there is no explosion and pointwise in case of explosion) and $Z^{i,l}\to Z^i$ for $1\leq i\leq m$ (with respect to the Skorohod $J_1$ topology when there is no explosion and in the uniform $J_1$ topology in case of explosion) as $l\to\infty$. 
(The convergence will be either weak or strong depending on the type of convergence of the $X^{i,\cdot, l}$ and $Y$.)
\end{theorem}
Note that the above limit theorem is either weak or strong, which follows from continuity properties of the multiparameter time change equations explored in Section \ref{stabilityAnalysisSection}. Indeed, we believe this is one strength of the time change representation versus, for example, the SDE representation which is found in the one-dimensional case in \cite{MR2584896} and \cite{MR3161481}. Indeed, even in the case of continuous sample paths, it is known that solving SDEs is a discontinuous operation of the driving processes. 
A manifestation of this is found in Wong-Zakai type phenomena (discovered in \cite{MR0183023}) and 
depending on the type of approximation to the driving processes one obtains limits to different SDEs,  as has been argued in \cite{courseRoughPathsFrizHairer}. 
On the other hand, Theorem \ref{stochasticLimitTheorem} does not depend on how one approximates the driving processes. We are not advocating, though, the use of one representation over another. The construction of \cite{MR3161481} is useful in the genealogical interpretation of continuous branching processes, constructing directly some of the flows in \cite{MR1990057}. 

From Theorem \ref{stochasticLimitTheorem} we deduce a limit theorem concerning multi-type Galton-Watson processes stated as Corollary \ref{GWCorollary}. In the one-dimensional case, Corollary \ref{GWCorollary} includes limit theorems found in \cite{MR0362529}, \cite{MR2225068} and \cite{MR3098685}. 
The multidimensional case has often been studied in the literature when the limit process is continuous, as in \cite{MR827331}.
We state a version without immigration, just to illustrate the kind of statement one can achieve as well as the technique. 
The technique can be adapted to the case of immigration as in Corollary 7 of \cite{MR3098685}. 
Let $(X^{1,\cdot, l}, 1\leq i\leq m)$ be independent $d$-dimensional random walks. Suppose that $X^{i,i,l}$ has jumps in $\z$ greater than $-1$ and that otherwise the coordinates have jumps in $\na$. 
Let $k^l=\paren{k^l_1,\ldots, k^l_m}\in \na^m$ be a sequence of starting states and define recursively the sequences 
$C^l=\paren{C^{j,l}, 1\leq j \leq m}$ and 
$Z^l=\paren{Z^{j,l}, 1\leq j \leq m}$ by
\begin{lesn}
C^l_0=0, 
\quad Z^l_0=k^l,
\quad Z^{j,l}_{n+1}=k^l+\sum_{i=1}^m X^{i,j,l}\circ C^{i,l}_n
\quad \text{and}\quad C^l_{n+1}=C^l_n+Z^l_{n+1}. 
\end{lesn}It is easy to see that for each $l$, $Z^l$ is a multitype Galton-Watson process 
such that the quantity of descendants of type $j$ of an individual of type $i$ 
has the same law as $X^{i,j,l}$ when $i\neq j$ and the law of $X^{i,i,l}+1$  in the remaining case. 
However, if $X^l$ is extended by constancy on intervals of the form $[n,n+1)$ with $n\in\na$, we see that $C^l$ is the Euler type approximation of span $1$ applied to $X^l$ that we have just introduced and $Z^l$ is the right-hand derivative of $C^l$. 

\begin{corollary}
\label{GWCorollary}
Let $X^{1,\cdot, l}$, $1\leq i\leq m$
 be independent $d$-dimensional random walks. 
Suppose that $X^{i,i,l}$ has jumps in $\z$ greater than $-1$ that otherwise the coordinates have jumps in $\na$. 

Assume that for each $i$ in $\set{1,\ldots,m}$ there are scaling constants $a_l$ and $b^i_l$ for $l\geq 1$ such that
\begin{lesn}
\paren{\frac{a_l}{b^j_l}X^{i,j,l}_{b^i_l t}, t\geq 0, 1\leq j\leq m}
\end{lesn}converges in Skorohod space (either almost surely or in distribution) to a L\'evy process $X^{i,\cdot}$. Furthermore, $a_l\to\infty$, $b^j_l/a_l\to \infty$ and $k^j_l$ is such that $k^j_l a_l/b^j_l\to z^j$.

Then, the scaled Galton-Watson processes
\begin{lesn}
\paren{\frac{a_l}{b_j}Z^{j,l}_{a_l t},t\geq 0, 1\leq j\leq m}
\end{lesn}%
started from $(k^{j}_l,1\leq j\leq m)$ converge in Skorohod space (either almost surely or in distribution) to the unique $\cb$ process $Z$ started from $z$ and constructed from $X$ and $Y=0$ in Theorem \ref{affineProcessRepresentationTheorem}. 
\end{corollary}
We end this section with an application of Corollary \ref{GWCorollary}. 
Note that the different processes in Corollary \ref{GWCorollary} have scalings that have to be adequately balanced in order to obtain a limit (with non-trivial reproduction and immigration components). 
In order to exemplify how this could be done, 
let us start by considering the framework of Theorem 4.2.2 of \cite{MR827331}, giving a limit theorem for nearly critical multitype Galton-Watson processes under finite-variance assumptions. 
Indeed, consider a sequence of multitype Galton-Watson processes $Z^{\cdot, l}$ such that $p^{i,l}$ is the law of the offspring of an individual of type $i$. We then define the mean matrix $M$ by means of\begin{lesn}
M^l_{i,j}=\sum_{k\in\na^m} k_j\imf{p^{i,l}}{k}. 
\end{lesn}Assume that $M^{l}=\id+C_l/l$ where $C_l\to C$ as $l\to\infty$. Consider also the variance matrix $\sigma^l$ given by\begin{lesn}
\sigma^l_{i,j}
=\bra{\sum_k \paren{k_j-M_{i,j}^l}^2\imf{p^{i,l}}{k}}^{1/2}.
\end{lesn}Supose that $\sigma^l_{i,j}\to \sigma_i\indi{i=j}$ as $l\to\infty$ and that the following Lindeberg condition holds:
\begin{lesn}
\sum_{k_i\geq \eps\sqrt{n} } \paren{k_i-M^l_{i,i}}^2\imf{p^{i,l}}{k}\to 0
\end{lesn}as $l\to\infty$. 
Recall our construction of such a process in terms of random walks $
X^{i,\cdot, l}$ for $i=1,\ldots, m$. 
From our hypotheses, it follows that $X^{i,i,l}_{l^2\cdot}/l$ converges to $\sigma_i B^i+C_{i,i}\id$ where $B^i$ is a standard Brownian motion. Indeed, the convergence of one-dimensional distributions is deduced from the Lindeberg-Feller central limit theorem. Because of independence and stationarity of the increments this implies the convergence of finite-dimensional distributions and tightness is easily deduced from the Aldous criterion. 
For $i\neq j$, one sees that $X^{i,j,l}_{l^2\cdot}/l\to C_{i,j}\id$ as $l\to\infty$. 
Indeed, it suffices again to establish convergence of one-dimensional distributions which follow from Chebyshev's inequality. Tightness again follows from the Aldous criterion. 
Hence, Corollary \ref{GWCorollary} allows us to conclude that if $Z^l_0/l\to z$ then $Z^l_{l\cdot}/l$ converges weakly to a continuous branching process $Z$ with continuous sample paths. 
One can then use the martingales associated to $X$, as in \cite{MR577310}, to see that the generator of $Z$ is given by\begin{lesn}
\sum_{i=1}^m \frac{z_i\sigma_i^2}{2}\frac{\partial^2 }{\partial z_i^2}+\sum_{1\leq i,j\leq m} z_iC_{i,j}\frac{\partial }{\partial z_j}. 
\end{lesn}This fact can also be deduced from the infinitesimal parameters of $Z$ that are introduced in Section \ref{preliminariesSection} and from the proof of Theorem \ref{affineProcessRepresentationTheorem}. 

Our work continues and extends the one-dimensional situation covered in \cite{MR3098685}. 
There are however, important differences with that work. 
First of all, the discussion of uniqueness to \eqref{affineProcessRepresentationEquation} now relies on the concept of (lack of) spontaneous generation. 
This is to be contrasted to the previous analysis based on taking inverses. 
The multiple time changes make this one-dimensional approach unfeasible. 
On the other hand, we also take the point of view of multiparameter time changes from \cite{MR577310}, providing a very concrete (but general) example of its applicability. 
This has led to several simplifications when proving that solutions to \eqref{affineProcessRepresentationEquation} are affine processes. 

The paper is organized as follows. 
We first consider a deterministic framework for equation \eqref{affineProcessRepresentationEquation} when $n=0$ and analyze existence, uniqueness, and basic measurability questions. 
This is done in Section \ref{deterministicPathwiseAnalysisSection}. 
We then undertake the proof of Theorem \ref{affineProcessRepresentationTheorem} when $n=0$, which reduces basically to establishing the Markov property and constructing relevant martingales, in Section \ref{nonNegativeAffineProcessesSection}. 
The case of general $n$ is taken up in Section \ref{generalConstructionSection}. 
Finally, we pass to the stability of equation \ref{affineProcessRepresentationEquation}, which contains the proofs of Theorem \ref{stochasticLimitTheorem} and Corollary \ref{GWCorollary} in Section \ref{stabilityAnalysisSection}. 

\section{Preliminaries on affine processes}
\label{preliminariesSection}

Let $Z$ be an affine process with laws $\paren{\p_z,z\in \st}$. 
Let $\Phi$ and $\psi$ be defined as in Equation \eqref{affineExponentEquation}; applying the Markov property, 
we get the semi-flow property
\begin{linenomath}
\begin{equation}
\imf{\psi}{t+s,u}=\imf{\psi}{s,\imf{\psi}{t,u}}\quad\text{and}\quad \imf{\Phi}{t+s,u}=\imf{\Phi}{t,u}\imf{\Phi}{s,\imf{\psi}{t,u}}. 
\end{equation}
\end{linenomath}
From Theorem 5.1 in \cite{MR2851694} or Theorem 3.3 in \cite{MR3040553}, it is known that the following derivatives exist and are continuous as a function of $u$:
\begin{linenomath}
\begin{esn}
\imf{F}{u}=\left.\frac{\partial }{\partial t}\imf{\Phi}{t,u}\right|_{t=0}
\quad \text{and}\quad
\imf{R}{u}=\left.\frac{\partial }{\partial t}\imf{\psi}{t,u}\right|_{t=0}. 
\end{esn}\end{linenomath}From the semi-flow property, we deduce the so called Riccati equations
\begin{linenomath}
\begin{equation}
\label{riccatiEquations}
\frac{\partial }{\partial t}\imf{\Phi}{t,u}=\imf{\Phi}{t,u}\imf{F}{\imf{\psi}{t,u}}
\quad\text{and}\quad
\frac{\partial }{\partial t}\imf{\psi}{t,u}=\imf{R}{\imf{\psi}{t,u}}
\end{equation}
\end{linenomath}with the initial conditions
\begin{linenomath}
\begin{esn}
\imf{\Phi}{0,u}=1\quad\text{and}\quad 
\imf{\psi}{0,u}=u. 
\end{esn}
\end{linenomath}If $\Phi$ were non-zero and $\phi$ were continuous and satisfied $e^{\phi}=\Phi$, we would obtain the more familiar equation
\begin{linenomath}
\begin{esn}
\frac{\partial }{\partial t}\imf{\phi}{t,u}=\imf{F}{\imf{\psi}{t,u}}
\end{esn}
\end{linenomath}which gives
\begin{linenomath}
\begin{esn}
\imf{\phi}{t,u}=\int_0^t \imf{F}{\imf{\psi}{s,u}}\, ds. 
\end{esn}
\end{linenomath}

Furthermore, Theorem 2.7 in \cite{MR1994043} asserts that $F$ and $R$ have the following very specific form: if $X^1,\ldots, X^m$ and $Y$ are L\'evy processes satisfying the conditions of Theorem \ref{affineProcessRepresentationTheorem} then  $F$ and $R=\paren{R_1,\ldots, R_{m+n}}$ are the unique continuous functions such that
\begin{linenomath}
\begin{equation}
\label{characteristicExponentsDefinition}
e^{\imf{F}{u}}=\esp{e^{u\cdot  Y_1}}
\quad\text{and}\quad
e^{\imf{R_i}{u}}= \esp{e^{u \cdot X^i_1}}
\end{equation}
\end{linenomath}for $1\leq i\leq m$ while for $m+1\leq i\leq n$ we set
\begin{lesn}
\imf{R_{m+i}}{u}=\sum_{j=1}^n \beta_{i,j} u_{m+j}. 
\end{lesn}

Furthermore, Section 6 of \cite{MR1994043} discusses the (global) existence and uniqueness of the generalized Riccati equations of \eqref{riccatiEquations}. 

\section{Pathwise analysis of the multidimensional time change equation}
\label{deterministicPathwiseAnalysisSection}
Following \cite{MR3098685}, we begin by considering a deterministic system of time change equations appearing in Theorem \ref{affineProcessRepresentationTheorem} in the case of non-negative processes ($n=0$). 
Consider $m\paren{m+1}$ \cadlag\ functions labeled $\set{f^{i,j}, 1\leq i,j\leq m}$ and $\set{g^{j}, 1\leq j\leq m}$. 
These functions satisfy the following requirements: 
\begin{description}
\label{hypothesisH}
\item[H1] $f^{i,j}$ has no negative jumps if $i=j$ and is non-decreasing otherwise.
\item[H2] $g^{j}$ is non-decreasing.
\item[H3] $\imf{g^j}{0}+\sum_{i=1}^m \imf{f^{i,j}}{0}\geq 0$ for $1\leq j\leq m$. 
\end{description}
The above hypotheses are collectively denoted \defin{H}. 

We seek a solution to the following system of equations for the \cadlag\ function $h=\paren{h^1,\ldots, h^m}$:
\begin{equation}
\label{odeSystem}
\imf{h^j}{t}=\sum_{i=1}^m \imf{f^{i,j}\circ c^i}{t}+\imf{g^j}{t}\quad\text{for $1\leq j\leq m$}\quad
\quad\text{where}\quad \imf{c^j}{t}=\int_0^t \imf{h^j}{s}\, ds. 
\end{equation}
This system can also be thought of as an ordinary differential equation for $c$ when one notes that $h^j$ is the right-hand derivative of $c^j$. 
With this interpretation, we might want to use other initial conditions for $c$ rather than only zero. 
This amounts to shifting the functions $f^{i,j}$; note however, that the shifts must still satisfy \defin{H3}. 
\subsection{A basic monotonicity lemma and existence}
Our approach to the study of \eqref{odeSystem} is based on its monotonicity properties. 
We begin with a simple and useful case of this and postpone an elaboration of this idea which will be useful to obtain uniqueness. 
\begin{lemma}
\label{firstMonotonicityLemma}
Suppose that we have two sets of functions $P=(f^{i,j},g^j)$ and $\tilde P=(\tilde f^{i,j},\tilde g^j)$ satisfying hypothesis \defin{H}. 
Assume that $f^{i,j}\leq \tilde f^{i,j}$ and $g^j\leq \tilde g^j$ and that additionally, 
 for every $j\in\set{1,\ldots, m}$, either $f^{i,j}< \tilde f^{i,j}$ or $g^j< \tilde g^j$. 
If $h$ and $\tilde h$ are non-negative functions that satisfy \eqref{odeSystem} driven by $P$ and $\tilde P$ respectively then $c\leq \tilde c$. 
\end{lemma}
\begin{proof}
Let
\begin{lesn}
P=\paren{f^{i,j},g^j,1\leq i,j\leq m}
\quad\text{and}\quad
\tilde P=\paren{\tilde f^{i,j},\tilde g^j,1\leq i,j\leq m}
\end{lesn}be a system of functions satisfying the assumptions of our lemma and $h, \tilde h$ the associated non-negative solutions to \eqref{odeSystem}. 
For any $\eps>0$ and any $\alpha>1$, define $\imf{\overline c^j}{t}=\imf{\tilde c^j}{\eps+\alpha t}$. 
Hence $\overline c^j$ has a \cadlag\ right-hand derivative $\overline h^j$ given by $\imf{\overline h^j}{t}=\alpha\imf{\tilde h^j}{\eps+\alpha t}$. 
We then define
\begin{linenomath}
\begin{esn}
\tau=\inf\set{t\geq 0: \imf{c^j}{t}>\imf{\overline c^j}{t}\text{ for some }j}
\end{esn}\end{linenomath}as well as the set $J$ of indices $j\in\set{1,\ldots, m}$ such that $c^j$ exceeds $\overline c^j$ strictly at some point of any right neighbourhood of $\tau$. If $j\in J$ then $\imf{c^j}{\tau}=\imf{\overline c^j}{\tau}$ while 
$\imf{c^i}{\tau}\leq \imf{\overline c^i}{\tau}$ for $i\neq j$, so that also $\imf{\tilde f^{i,j}\circ c^i }{\tau}\leq \imf{\tilde f^{i,j}\circ \overline c^i }{\tau}$ for $i\neq j$. 
We deduce the following for $j\in J$: 
\begin{linenomath}
\begin{align*}
0\leq \imf{h^j}{\tau}
&=
\sum_{i} \imf{f^{i,j}\circ c^i}{\tau}+\imf{g^j}{\tau}
\\&<
\sum_{i} \imf{\tilde f^{i,j}\circ c^i}{\tau}+\imf{\tilde g^j}{\tau} 
\\&<\alpha\bra{
\sum_{i} \imf{\tilde f^{i,j}\circ \overline c^i}{\tau}+\imf{\tilde g^j}{\eps+\alpha\tau}}=\imf{\overline h^j}{\tau}.
\end{align*}\end{linenomath}(Note that the right-hand side of the first strict inequality cannot be zero, which justifies the second strict inequality.) 
We deduce that $c^j$ remains below $\overline c^j$ in a right neighbourhood of $\tau$ which contradicts the definitions of $\tau$ and $J$. 
We deduce that $c^j\leq \overline c^j$ and, letting $\alpha$ go to $1$, that $c^j\leq \tilde c^j$. 
\end{proof}

We now tackle existence for \eqref{odeSystem} in the case when only $f^{j,j},1\leq j\leq m$ are not piecewise constant. 
The proof will be based on the observation that under the piecewise constant hypotheses, the system \eqref{odeSystem} is one-dimen\-sional on adequate intervals. 
The piecewise constant case will allow us to prove existence for \eqref{odeSystem} in general through the monotonicity proved in Lemma \ref{firstMonotonicityLemma}. 

\begin{lemma}
\label{piecewiseConstantExistenceLemma}
Let $\set{f^{i,j}, g^j,1\leq i,j\leq m}$ satisfy \defin{H}
and suppose that $f^{i,j}$ and $g^j$ are piecewise constant if $1\leq i,j\leq m$ and $i\neq j$. 
Then, there exists a solution $h=\set{h^j:1\leq j\leq m}$ to \eqref{odeSystem}%
. 
This solution exists on an interval $[0,\tau)$ and $c^j_{\tau-}=\infty$ for some $j$. 
\end{lemma}
The time $\tau$ is termed the explosion time of $c$. 

\begin{remark}
\label{measurabilityForUnidimensionalExistenceLemmaRemark}
For the one-dimensional case, existence follows from Theorem 1 in \cite{MR3098685} which asserts that the problem $\imf{\ivp}{f,0,x}$ consisting of a finding a function $c$ with a right-hand derivative $h$ which satisfies
\begin{linenomath}
\begin{esn}
\imf{\ivp}{f,0,x}:\quad h=f\circ c\quad \text{with}\quad\imf{c}{0}=x
\end{esn}\end{linenomath}admits,  
for any $x\geq 0$ and any \cadlag\ function $f$ such that $f$ has no negative jumps, a unique solution which lacks spontaneous generation. 
When $f(x)=0$, the only solution lacking spontaneous generation is the function $c(t)=x$. 
When $f(x)>0$, the unique solution can be constructed by a Lamperti type transformation obtained by first making zero absorbing after $x$; formally
\begin{linenomath}\begin{esn}
T=\inf\set{t\geq x: \imf{f}{t}=0}
\quad\text{and}\quad
\imf{\tilde f}{t}=\begin{cases}
\imf{f}{t}& t<T\\
0&t\in [T,\infty]
\end{cases}. 
\end{esn}\end{linenomath}We then define $i$ on $[x,\infty)$ by means of
\begin{linenomath}\begin{esn}
\imf{i}{y}=\int_x^y \frac{1}{\imf{\tilde f}{t}}\, dt. 
\end{esn}\end{linenomath}Note that $i$ is strictly increasing on $[x,T)$ and infinite on $(T,\infty)$. 
Then, let $c$ be the right-continuous inverse of $i$ (in the sense of Lemma 0.4.8 of \cite{MR1725357}). 
Note that $c$ is strictly increasing on $[0, i(T-)]$ and constant on $[I(T-),\infty]$ and by definition $c(0)=x$. 
Then, since the right-hand derivative of $i$ exists and equals $1/f$, then $c$ also admits a right-hand derivative (on $[0,\imf{i}{T-} )$), say $h$, and we have $h=1/(1/f\circ i^{-1})=f\circ c$. 
The function $c$ so constructed from $f$ is called the Lamperti transform of $f$ absorbed at its first zero after $x$. 
Note that $\imf{c}{\infty}=T$. 
In the one-dimensional setting, when $X$ is a spectrally positive L\'evy process, Proposition 2 of \cite{MR3098685} shows that there is a unique solution $C$ to $\imf{\ivp}{x+X,0, 0}$ (with right-hand derivative $Z$) which has zero as an absorbing state; if $T$ denotes the hitting time of zero of $x+X$, $\tilde X$ equals $X$ stopped at $T$, then $C$ is also the unique solution $\imf{\ivp}{x+X,0,0}$, so that $C_\infty=T$. 
This one dimensional result is important in our proof of uniqueness of solutions to \ref{affineProcessRepresentationEquation}. Since stopping a \cadlag\ process at a stopping time and looking at a \cadlag\ process at a random time are measurable transformations, we see that the Lamperti transformation is measurable on the Skorohod space of \cadlag\ trajectories with the $\sigma$-field generated by projections. This would hold even if we take the initial value $x$ to be random and measurable. 
\end{remark}

\begin{proof}
Suppose first that
\begin{linenomath}
\begin{esn}
f^{i,j}=\sum_{k=1}^\infty x^{i,j}_k\indi{[t^{i,j}_{k-1},t^{i,j}_k)}\quad g^j=\sum_{k=1}^\infty y^j_k \indi{[t^j_{k-1},t^j_k)},
\end{esn}\end{linenomath}where $x^{i,j}_{k-1}\leq x^{i,j}_{k}$ if $i\neq j$, $0=t^{i,j}_0\leq t^{i,j}_1\leq \cdots$, the sequence $t^{i,j}_k,k\geq 0$ has no accumulation points (similar assumptions hold for $g^j$) and additionally, for each $j$
\begin{linenomath}\begin{esn}
\imf{f^{j,j}}{0}+\sum_{i\neq j}x^{i,j}_1+y^j_1\geq 0
\end{esn}\end{linenomath}so that assumptions \defin{H} hold. 
Let $T_{i,j}$ (resp. $T_j$) denote the set of change points of the functions $f^{i,j}$ (resp. $g^j$):
\begin{linenomath}
\begin{esn}
T_{i,j}=\set{t^{i,j}_k:0\leq k}\quad\text{and} \quad T_j=\set{t^j_k:0\leq k}. 
\end{esn}\end{linenomath}Let $\tau_0=0$ and, for any $j=1,\ldots, m$,  let $\tilde c^j_1$ be the unique solution of the problem $\imf{\ivp}{f^j_1,0 , 0}$, where the function $f^j_1$ is given by
\begin{linenomath}
\begin{esn}
\imf{f^j_1}{t}=\imf{f^{j,j}}{t}+\sum_{i\neq j} \imf{f^{i,j}}{0}+\imf{g^j}{0}. 
\end{esn}
\end{linenomath}We now define the times
\begin{linenomath}
\begin{esn}
\tau^i_1=\inf\set{t>0: t\in T_i\text{ or } \imf{\tilde c^i_1}{t}\in \cup_{j\neq i}T_{i,j}}
\quad\text{and}\quad 
\tau_1=\min_i \tau^i_1. 
\end{esn}
\end{linenomath}Set $c^j$ equal to $\tilde c^j_1$ on $[0,\tau_1]$ and recursively define $\tilde c^j_{n+1}$ as the solution to $\imf{\ivp}{f^j_{n+1},0,\imf{c^j}{\tau_n}}$ where the function $f^j_{n+1}$ is given by
\begin{linenomath}
\begin{esn}
\imf{f^j_{n+1}}{t}=\imf{f^{j,j}}{t}+\sum_{i\neq j} f^{i,j}\circ \imf{c^i}{\tau_n}+\imf{g^j}{\tau_n}. 
\end{esn}
\end{linenomath}We then define
\begin{linenomath}
\begin{esn}
\tau^i_{n+1}=\inf\set{t>\tau_n: t\in T_i\text{ or }\imf{\tilde c^i_1}{t-\tau_n}\in \cup_{j\neq i}T_{i,j}}
\quad\text{and}\quad
\tau_{n+1}=\min_{i}\tau^i_{n+1}
\end{esn}
\end{linenomath}and let $\imf{c^j}{t}=\imf{\tilde c^j_{n+1}}{t-\tau_{n}}$ on $[\tau_n,\tau_{n+1}]$. 
We assert that $c=\paren{c^1,\ldots, c^m}$ solves \eqref{odeSystem}; the proof is by induction. 
However, note that the starting point of $\tilde c_n$ is chosen so that $c$ is continuous and has a \cadlag\ right-hand derivative. 
On $[0,\tau_1]$, $f^{i,j}\circ c^i$ and $g^j$ are constant and hence, equal to their value at zero. 
Hence, if we let $h^j$ stand for the right-hand derivative of $c^j$, we obtain the following equalities for any $t<\tau_1$
\begin{linenomath}
\begin{align*}
\imf{h^j}{t}
&=\imf{f^{j}_1\circ c^j}{t}
\\&=\imf{f^{j,j}\circ c^j}{t}+\sum_{i\neq j} \imf{f^{i,j}}{0}+\imf{g^j}{0}
\\&=\imf{f^{j,j}\circ c^j}{t}+\sum_{i\neq j} \imf{f^{i,j}\circ c^i}{t}+\imf{g^j}{t},
\end{align*}
\end{linenomath}which allow us to conclude that $c$ solves \eqref{odeSystem} on $[0,\tau_1]$. 
On the other hand, if we assume that $c$ solves \eqref{odeSystem} on $[0,\tau_n]$, then note that, by definition, $f^{i,j}\circ c^i$ and $g^j$ are constant on $[\tau_n,\tau_{n+1}]$. 
We deduce that for $t\in [\tau_n,\tau_{n+1}]$:
\begin{linenomath}
\begin{align*}
\imf{h^j}{t}
&=f^{j}_{n+1}\circ \imf{\tilde c^j_{n+1}}{t-\tau_n}
\\&=f^{j,j}\circ \imf{\tilde c^j_{n+1}}{t-\tau_n}+\sum_{i\neq j} f^{i,j}\circ \imf{c^j}{\tau_n}+g^j\circ \imf{c^j}{\tau_n}
\\&=f^{j,j}\circ \imf{c^j}{t}+\sum_{i\neq j} f^{i,j}\circ \imf{c^i}{t}+\imf{g^j}{t}
\end{align*}\end{linenomath}so that $c$ solves \eqref{odeSystem} on $[0,\tau_{n+1}]$.

Since $\tau_n$ increases in $n$, there are two possibilities: either $\tau_n\to \infty$ (in which case the solution we have constructed is a global solution) or $\imf{c^j}{\tau_n}\to \infty$  for some $j$ by definition of $\tau_n^j$, $\tau_n$, and the fact that the sets $T_{i,j}$ and $T_j$ have no accumulation points. 
In the latter case, $c$ explodes. 
\end{proof}

\begin{remark}
\label{measurabilityForMultidimensionalExistenceLemmaRemark}
As in Remark \ref{measurabilityForUnidimensionalExistenceLemmaRemark}, we note that if we apply the procedure of the above proof in the case of \cadlag\ stochastic processes (satisfying the conditions of Lemma \ref{piecewiseConstantExistenceLemma}) then the solutions are measurable. This follows because on adequate intervals (which are obtained by stopping), the solutions are unidimensional and are constructed through the Lamperti transformation. 
\end{remark}

We now tackle existence for \eqref{odeSystem}. 
\begin{lemma}
\label{generalExistenceLemma}
Let $\set{f^{i,j},g^j, 1\leq i,j\leq m}$ satisfy \defin{H}. 
Then, there exists $\tau>0$ such that a non-negative solution $h$ to \eqref{odeSystem} exists on $[0,\tau)$. 
Furthermore this solution explodes at $\tau$ and is maximal:
\begin{enumerate}
\item If $\imf{c^j}{t}=\int_0^t \imf{h^j}{s}\, ds$ then $\imf{c^j}{\tau-}=\infty$ for some $j$. 
\item If $\tilde h$ is another solution to \eqref{odeSystem} (with its corresponding $\tilde c$) then $\tilde c\leq c$ on the interval of existence of $\tilde h$.  
\end{enumerate}
\end{lemma}
\begin{proof}
For $1\leq i,j\leq m$ with $i\neq j$ consider a sequence of \cadlag\ functions $f^{i,j}_n$ and $g^j_n$  which are piecewise constant, are strictly bigger than $f^{i,j}$ and $g^j$,  and decrease as $n\to \infty$ towards $f^{i,j}$ and $g^j$ respectively. 
We then set $f^{j,j}_n=f^{j,j}$. 
Using Lemma \ref{piecewiseConstantExistenceLemma}, we can consider for any $n$ a solution $h_n=(h^1_n,\ldots, h^m_n)$ to \eqref{odeSystem} driven by $\{f^{i,j}_n,g^j_n\}$
. 
By Lemma \ref{firstMonotonicityLemma}, we see that the cumulative population of $h_n$ exceeds the cumulative population of any solution to \eqref{odeSystem}. 

Fix any $K>0$ and use it to stop $c_n$ at the instant $\tau_{n,K}$ that any one of its coordinates reach $K$. 
Call the resulting function $\tilde c_n$. 
Since $c_{n+1}\leq c_n$ then $\tau_{n,K}\leq \tau_{n+1,K}$; set $\tau_K=\lim_n \tau_{n,K}$. 
Note that $\tilde c_n$ has a \cadlag\ derivative $\tilde h^n$ given by
\begin{linenomath}\begin{esn}
\imf{\tilde h^j_n}{t}=\indi{t\leq \tau_{n,K}} \imf{h^j_n}{t}
. 
\end{esn}\end{linenomath}Hence,  $\tilde h^j_n$ can be bounded on any interval $[0,t]$ by $m\max_{i,j}\inf_{x\leq K}\imf{f^{i,j}_n}{x}+\imf{g^j_n}{t}$, and can then be bounded in $n$ by construction of $f^{i,j}_n$ and $g^j_n$. 
By the Arzel\`a-Ascoli theorem (which applies since $\imf{\tilde c^j_n}{0}=0$), $\tilde c_n$ is sequentially compact. 
We now show that every subsequential limit coincides. 
Indeed, if $\tilde c$ is the (uniform) limit (on compact sets) of $\tilde c^j_{n_k}$ as $k\to\infty$, 
then the bounded convergence theorem implies that for any $t<\tau_K$:
\begin{linenomath}
\begin{esn}
\imf{\tilde c^j}{t}
=\lim_{k}\imf{\tilde c^j_k}{t}
=\lim_{k}\int_0^t \sum_i \imf{f^{i,j}_n\circ \tilde c^i_n}{s}+\imf{g^j_n}{s}\, ds
= \int_0^t \sum_i \imf{f^{i,j}\circ c^i}{s}+\imf{g^j}{s}\, ds.
\end{esn}%
\end{linenomath}
We conclude that $\tilde c$ admits a right-hand derivative $\tilde h$ on $[0,\tau_K)$ which satisfies \eqref{odeSystem} on $[0,\tau_K)$. 
However, $\tilde c$ is the maximal solution by construction (since we can apply Lemma \ref{firstMonotonicityLemma} to the approximations $c_n$), so that all subsequential limits agree on $[0,\tau_K]$. 
Finally, note that before $\tau_K$ the coordinates of $\tilde c$ have to be smaller than $K$ and that at $\tau_K$ some coordinate equals $K$. 
Hence $\tau_K$ coincides with the instant in which some coordinate of $\tilde c$ reaches $K$. 
By uniqueness, one can construct a function $c$ which coincides with $\tilde c$ on $[0,\tau_K)$, so that $c$ is defined and solves \eqref{odeSystem} on $[0,\tau)$ where $\tau=\lim_K\tau_K$. 
By construction, $c$ explodes at $\tau$ and is maximal in the class of solutions to \eqref{odeSystem}. 
%
%
%
\end{proof}
\begin{remark}
\label{measurabilityForGeneralExistenceLemmaRemark}
Recall that the approximations of the above proof are measurable in the case of applying them to \cadlag\ stochastic processes thanks to Remark \ref{measurabilityForMultidimensionalExistenceLemmaRemark}. 
Then, applying the construction to a \cadlag\ stochastic process $X$ satisfying hypotheses \defin{H}, we get another pair of stochastic processes $Z$ and $C$. 
Since $Z$ and $C$ are cadlag, then $X^{i,j}\circ C^i$ is also a stochastic process. 
\end{remark}

\subsection{Spontaneous generation and minimal solutions}
An interpretation for the one-dimensional case of \eqref{odeSystem} was proposed in \cite{MR3098685} by noting that if $f^{1,1}$ represents the breadth-first walk on a (combinatorial) forest representing the genealogy of a population with immigrants along each generation and $g^{1}$  codes the immigration to the population then $h^1$ is the population profile (that is, the sequence of generation sizes), while $c^1$ is the cumulative population. 
The multidimensional case of this discrete coding can be found in Subsection 2.2 of \cite{Chaumont:2013aa}, when $g^j=0$ for all $j$, and it shows that the one-dimensional interpretation still holds. 
In particular, the discrete interpretation gives sense to the following definition of lack of spontaneous generation: in the one dimensional case, a solution $h=f\circ c+g$ lacks spontaneous generation if $\imf{h}{s}=0$ (the population is zero at time $s$) and $g$ is constant on $[s,t]$ (there is no immigration on $[s,t]$) implies that $h=0$ on $[s,t]$ (the population remains at zero). 
Perhaps it is then surprising that there are solutions featuring spontaneous generation, but when $f$ is the typical path of a normalized Brownian excursion then they do exist (cf. \cite[Sect. 2]{MR3098685}). 
\begin{definition}
Let $\paren{f^{i,j},g^j}$ satisfy \defin{H}. 
We say that a solution $h=\paren{h^j}$ to \eqref{odeSystem} has no spontaneous generation if whenever $\imf{h^j}{s}=0$ for some $s\geq 0$ and for all $j$ in $J\subset\set{1,\ldots, m}$, we have that the strict increase of $c^j$ at $s$ for some $j\in J$ implies that either $g^j$ increases strictly at $s$ or there exists $i\not\in J$ such that $f^{i,j}\circ c^i$ increases strictly to the right of $s$. 
\end{definition}

As a remark, we mention that $h$ lacks spontaneous generation if and only if at any $s\geq 0$ such that the set $\imf{J}{s}=\set{j:\imf{h^j}{s}=0}$ is nonempty, the strict increase of $c^j$ at $s$ for some $j\in \imf{J}{s}$ implies that either $g^j$ increases strictly to the right of $s$ or there exists $i\not\in\imf{J}{s}$ such that $f^{i,j}$ increases strictly to the right of $\imf{c^i}{s}$. 

The definition works very well with induction on the dimension, in the sense that if $h=\paren{h^i,i\leq m}$ is a non-negative solution to \eqref{odeSystem} driven by $\paren{f^{i,j},1\leq i,j\leq m}$ and $\paren{g^j,j\leq m}$ without spontaneous generation and $m_1<m$,  we can then consider $h^1,\ldots, h^{m_1}$ as a solution, which will lack spontaneous generation, to \eqref{odeSystem} but driven by $f^{i,j}$ and $g^j+\sum_{i>m_1}f^{i,j}\circ c^i$ for $1\leq j\leq m_1$.

The importance of solutions lacking spontaneous generation is that they have monotonicity properties (see Lemma \ref{monotonicityLemma} below) and, consequently, they are minimal solutions to \eqref{odeSystem} as well as unique. 
In particular, if all solutions of \eqref{odeSystem} can be shown to have no spontaneous generation, then there is at most one solution. 
There are two cases when we can actually apply this technique.
First, when $g^j$ is strictly increasing for all $j$ since then solutions trivially have no spontaneous generation. 
Another example is when \eqref{odeSystem} is driven by L\'evy processes satisfying the hypotheses of Theorem \ref{affineProcessRepresentationTheorem}: we will show in Lemma  \ref{noSpontaneousGenerationForLevyProcessesLemma} that solutions have no spontaneous generation, which covers the uniqueness statement in Theorem \ref{affineProcessRepresentationTheorem}. 
\begin{lemma}
\label{monotonicityLemma}
Suppose that we have two sets of functions $P=(f^{i,j},g^j)$ and $\tilde P=(\tilde f^{i,j},\tilde g^j)$ satisfying hypothesis \defin{H}.
Assume that $f^{i,j}\leq \tilde f^{i,j}$ and $g^j\leq \tilde g^j$. 
If $h$ and $\tilde h$ are non-negative functions that satisfy \eqref{odeSystem} driven by $P$ and $\tilde P$ respectively and $h$ lacks spontaneous generation, then $c\leq \tilde c$. 
Hence, \eqref{odeSystem} admits at most one solution $h$ whose coordinates are non-negative and have no spontaneous generation. 
\end{lemma}
The above lemma also tells us that solutions without spontaneous generation are minimal in the sense that their primitive is a lower bound for the primitive of any other solution. 
\begin{proof}
This proof is an elaboration of the proof of Lemma \ref{firstMonotonicityLemma}. 
We proceed by induction. 
Let $m=1$, let $P=\paren{f,g}$  and $\tilde P=\paren{\tilde f,\tilde g}$ satisfy hypothesis \defin{H} and let $h$ and $\tilde h$
be non-negative solutions to \eqref{odeSystem} driven by $P$ and $\tilde P$ and lacking spontaneous generation. 
Let $\eps>0$ and $\alpha>1$ and use them to define $\overline c$ by means of $\imf{\overline c}{t}=\imf{\tilde c}{\eps+\alpha t}$. 
Let
\begin{linenomath}
\begin{esn}
\tau=\inf\set{t\geq 0: \imf{c}{t}>\imf{\overline c}{t}}.
\end{esn}\end{linenomath}

Note that $\overline c$ has a \cadlag\ right-hand derivative $\overline  h$ given by $\imf{\overline h}{t}=\alpha \imf{\tilde h}{\eps+\alpha t}$. 
If $\imf{\overline h}{\tau}>0$, since $\imf{c}{\tau}=\imf{\overline c}{\tau}$, 
our assumptions give
\begin{linenomath}
\begin{align*}
\imf{h}{\tau}
&=\imf{f\circ c}{\tau}+\imf{g}{\tau}
\\&\leq \imf{\tilde f\circ \overline c}{\tau}+\imf{g}{\tau}
\\&<\alpha\bra{\imf{\tilde f\circ \overline c}{\tau}+\imf{\tilde g}{\eps+\alpha\tau}}=\imf{\overline h}{\tau}.
\end{align*}
\end{linenomath}Hence, $c\leq \overline c$ on a right neighbourhood of $\tau$ contradicting its definition. 
On the other hand, if $\imf{\overline h}{\tau}=0$, 
then we can only infer, as in the previous display, that
\begin{linenomath}
\begin{align*}
0
\leq \imf{h}{\tau} 
&=\imf{f\circ c}{\tau}+\imf{g}{\tau}
\\&\leq 
\imf{f\circ c}{\tau}+\imf{g}{\eps+\alpha\tau}
\\&\leq\imf{\tilde f\circ \tilde c}{\tau}+\imf{\tilde g}{\eps+\alpha\tau}
=0.
\end{align*}
\end{linenomath}We conclude that $g$ is constant on $[\tau,\eps+\alpha\tau]$ which, by lack of spontaneous generation, shows that $h=0$ on the same interval so that $c$ cannot exceed $\tilde c$ in any small enough right neighbourhood of $\tau$. 
We conclude that $\imf{c}{t}\leq \imf{\tilde c}{\eps+\alpha\tau}$ for any $t\geq 0$, any $\eps>0$ and any $\alpha>1$. Hence $c\leq \tilde c$. 

Let $m\geq 2$. 
Suppose now that the monotonicity statement $c\leq \tilde c$ is true for any solution to \eqref{odeSystem} of dimension strictly less than $m$. 
Let $P=\paren{f^{i,j},g^j,1\leq i,j\leq m}$ and $\tilde P=\paren{f^{i,j},g^j,1\leq i,j\leq m}$ be a system of functions satisfying the assumptions of our lemma in dimension $m$ and $h$ and $\tilde h$ the associated non-negative solutions to \eqref{odeSystem} without spontaneous generation. 
We proceed as in the one dimensional case: for any $\eps>0$ and any $\alpha>1$, define $\imf{\overline c^j}{t}=\imf{\tilde c^j}{\eps+\alpha t}$. 
Hence $\overline c^j$ has a \cadlag\ right-hand derivative $\overline h^j$ given by $\imf{\overline h^j}{t}=\alpha\imf{\tilde h^j}{\eps+\alpha t}$. 
We then define
\begin{linenomath}
\begin{esn}
\tau=\inf\set{t\geq 0: \imf{c^j}{t}>\imf{\overline c^j}{t}\text{ for some }j}
\end{esn}\end{linenomath}as well as the set $J$ of indices $j\in\set{1,\ldots, m}$ such that $c^j$ exceeds $\overline c^j$ strictly at some point of any right neighbourhood of $\tau$. If $j\in J$ then $\imf{c^j}{\tau}=\imf{\overline c^j}{\tau}$ while 
$\imf{c^i}{\tau}\leq \imf{\overline c^i}{\tau}$ for $i\neq j$. 
If $\imf{\overline h^j}{\tau}>0$ for some $j\in J$, we infer that
\begin{linenomath}
\begin{align*}
\imf{h^j}{\tau}
&=\imf{f^{j,j}\circ c^j}{\tau}+\sum_{i\neq j} \imf{f^{i,j}\circ c^i}{\tau}+\imf{g^j}{\tau}
\\&<\alpha\bra{\imf{\tilde f^{j,j}\circ \overline c^j}{\tau}+\sum_{i\neq j} \imf{\tilde f^{i,j}\circ \overline c^i}{\tau}+\imf{\tilde  g^j}{\eps+\alpha\tau}}=\imf{\overline h^j}{\tau}.
\end{align*}\end{linenomath}We deduce that $c^j$ remains below $\overline c^j$ in a right neighbourhood of $\tau$ which contradicts the definitions of $\tau$ and $J$. 
Hence, we can assume that $\imf{\overline h^j}{\tau}=0$ for every $j\in J$. 
Note that if $J=\set{1,\ldots, m}$ then
\begin{linenomath}\begin{align*}
0
\leq \imf{h^j}{\tau}
&=\sum_{i} \imf{f^{i,j}\circ c^i}{\tau}+ \imf{g^j}{\tau}
\\&\leq \sum_{i} \imf{f^{i,j}\circ c^i}{\tau}+ \imf{g^j}{\eps+\alpha\tau}
\\&\leq \sum_{i} \imf{\tilde f^{i,j}\circ \overline  c^i}{\tau}+ \imf{\tilde g^j}{\eps+\alpha\tau}
=0.
\end{align*}\end{linenomath}We conclude not only that $\imf{h^j}{\tau}=0$, but also that $g^j$ is constant on $[\tau,\eps+\alpha\tau]$, which implies, by lack of spontaneous generation, that $h^j$ is constant on $[\tau,\eps+\alpha\tau]$, which contradicts the definition of $\tau$. 
Hence, we can assume that $J\subsetneq \set{1,\ldots, m}$ and by relabelling, we write $J=\set{1,\ldots, m_1}$ where $m_1<m$. 
For every $j>m_1$, $c^j\leq \overline c^j$ in a right neighbourhood of $\tau$. 
Also, note that $c^j,j\leq m_1$ solves system \eqref{odeSystem} in dimension $m_1$ when driven by $\paren{f^{i,j} ,i,j\leq m_1}$ and $g^j+\sum_{i>m_1}f^{i,j}\circ c^j$. The same remark holds for $\overline c^j,j\leq m_1$. Since this system have dimension strictly less than $m$ and the reduced system $h_1,\ldots, h^{m_1}$ has no spontaneous generation, monotonicity holds for them and we can conclude that $c^j\leq \overline c^j, j\leq m_1$ in a right neighbourhood of $\tau$, which again contradicts the definition of $\tau$. As before, we conclude that $c\leq \tilde c$. 

Finally, suppose that $h$ and $\tilde h$ are two non-negative solutions to \eqref{odeSystem} which lack spontaneous generation and are driven by the same functions $P$. 
Applying our monotonicity statement, we see that $c=\tilde c$ which then implies $h=\tilde h$. 
\end{proof}

\subsection{Further consequences in the stochastic setting}
We now show that the process $C$ is a multiparameter random time change in the sense of \cite[Ch. 6]{MR838085}. 
For this, consider the \sa
\begin{linenomath}\begin{equation}
\label{multiparameterFiltrationDefinition}
\F^\circ_{t_1,\ldots, t_{m},t}=\F^{X^1}_{t_1}\vee\cdots\vee \F^{X^m}_{t_m}\vee \F^Y_{t}. 
\end{equation}\end{linenomath}

\begin{lemma}
\label{measurabilityPropertiesLemma}
Let $X^1,\ldots, X^m$ and $Y$ be a stochastic process satisfying hypotheses \defin{H}. 
Let $Z$ be the solution to \eqref{affineProcessRepresentationEquation} (with $n=0$) such that its primitive $C$ is maximal. 
Then\begin{linenomath}
\begin{align*}
\sag{Z^j_s,C^j_s, Y^j_s,X^{i,j}\circ C^i_s: 0\leq s\leq t, 0\leq i,j\leq m}\cap\set{C^i_t\leq t_i,1\leq i\leq m}
\subset \F^\circ_{t_1,\ldots, t_m,t}.
\end{align*}\end{linenomath}
\end{lemma}
Recall that the solution constructed in Lemma \ref{generalExistenceLemma} has a maximal primitive. 
The proof will be based on a Galmarino type test in the multiparameter setting. (Cf. \cite{MR0174082} and \cite[Ex. 1.4.21]{MR1725357}.) 
\begin{proof}
Let $Z$, $C$ be as in the statement. 
Consider also the solution $\tilde Z$ to \eqref{affineProcessRepresentationEquation} (with $m=0$), but now driven by $X^i$ stopped at $t_i$ (denoted $\tilde X^i$) for $1\leq i\leq m$, by $Y$ stopped at $t$ (denoted $\tilde Y$), and such that its primitive $\tilde C$ is maximal. 
Analysing the construction of $C$, $Z$, $\tilde C$ and $\tilde Z$ we note that if $C^i_t\leq t_i$ for $1\leq i\leq m$ then $Z=\tilde Z$, $X^{i,j}\circ C^i=\tilde X^{i,j}\circ \tilde C^i$ and $C=\tilde C$  on $[0,t]$. 
Since $\tilde Z_t$, $\tilde X^{i,j}\circ \tilde C^i$ and $\tilde C^i$ are measurable functions of $\tilde X^{i,j}$, the statement follows. 
\end{proof}

We now study the uniqueness of \eqref{affineProcessRepresentationEquation} when $n=0$. 
\begin{lemma}
\label{noSpontaneousGenerationForLevyProcessesLemma}
Let $X^1,\ldots, X^m$ and $Y$ be L\'evy processes satisfying the assumptions of Theorem \ref{affineProcessRepresentationTheorem} when $n=0$. 
Then, almost surely, solutions $Z$ to \eqref{affineProcessRepresentationEquation} have no spontaneous generation. 
\end{lemma}

\begin{proof}
Let $Z$ (equivalently $C$) be a solution to \eqref{affineProcessRepresentationEquation}. 
Let $\tau$ be the first instant such that $Z$ admits spontaneous generation. 
We argue that $\tau=\infty$ by contradiction. 
Indeed, we first show that $\tau>0$ almost surely and then we apply arguments related to the strong Markov property to deduce that $\tau<\tau$ on the set $\tau<\infty$.

Let us now show that $\tau>0$. 
Note that $\tau>0$ means that if $J=\set{j\leq m: z_j=0}$ then the assumption that $Y^j=0$ on a right neighbourhood of $0$ for every $j\in J$ and that $X^{i,j}=0$ on a right neighbourhood of zero for every $j\in J$ and $i\not\in J$ implies that $Z^j=0$ for every $j\in J$ on a right neighbourhood of zero. 
Hence, the problem is reduced to proving that if $z_j=0$  and $Y^j=0$ for every $j$ then $Z=0$. 
(This corresponds to analyzing the case of multitype CB processes without immigration. )
Let $T>0$ be such that $C^j_T<\infty$ for every $j$. 
Since $X^{i,j}$ is a subordinator for any $i\neq j$ then $\lim_{h\to 0}X^{i,j}_h/h$ exists and equals the drift coefficient of $X$ (cf. \cite[Prop 8, Ch III]{MR1406564}). 
Hence, there exists $M>0$ such that $X^{i,j}<M\id$ on $[0,C^j_T]$ for all $i\neq j$. 
If $j$ is any coordinate such that $X^{j,j}$ is a finite variation L\'evy process (that is, the difference of two subordinators), the same argument implies that (for a possibly different $M$) $X^{j,j}< M \id$ on $[0,C^j_T]$. 
For this coordinate we see that
\begin{linenomath}
\begin{esn}
0\leq Z^j_t\leq M C^j_t+M\sum_{i\neq j} C^{i}_t.
\end{esn}\end{linenomath}From Gronwall's inequality, we see that $Z$ equals zero until there exists $i\neq j$ such that $C^i$ grows. 
Hence, coordinates $j$ such that $X^{j,j}$ have finite variation cannot be responsible for spontaneous generation. 
To analyze coordinates with infinite variation, recall from \cite{MR0242261} that if $X^{j,j}$ has infinite variation then $\liminf_{h\to 0}X^{j,j}_h/h=-\infty$. 
We consider $m\geq 2$ since the case $m=1$ has been handled in \cite{MR3098685}. 
In particular, for  $A>m-1$, we can choose a sequence $(t_n)$ decreasing to zero and such that $X_{t_n}^{j,j}<-M At_n$. 
We then choose $\eps_n$ in the interval $(M t_n,M A t_n/(m-1))$. 
Now, define $\tilde X^{i,j,n}= \eps_n\vee \paren{ M\id}$ for $i\neq j$ and consider a solution $\tilde Z$ to
\begin{linenomath}\begin{esn}
\tilde Z^j_t = X^{j,j}\circ \tilde C^j_t+\sum_{i\neq j} \tilde X^{i,j,n}\circ \tilde C^i_t.
\end{esn}\end{linenomath}%
A modification of the proof of Lemma \ref{firstMonotonicityLemma} shows that $C^j\leq \tilde C^j$ on $[0,T]$. 
However, note that while every $\tilde C^j$ is below $\eps_n/M$, $\tilde C^j$ behaves as the solution to
\begin{linenomath}\begin{esn}
\tilde Z^j_t = X^{j,j}\circ \tilde C^j_t+(m-1)\eps_n.
\end{esn}\end{linenomath}
It follows that if $\tilde Z^j$ reaches zero before any $\tilde C^i$ exceeds $\eps_n/M$ then $\tilde Z^j$ remains at zero afterwards (since this happens for the one-dimensional problem defining $\tilde Z^j$).
However, recall from Remark \ref{measurabilityForUnidimensionalExistenceLemmaRemark}, that in the one-dimensional case the total population ($C^j_\infty$) equals the time the reproduction function reaches zero ($\inf\{t\geq 0: (m-1)\eps_n+ X^{j,j}_t=0\}$). 
Since $\paren{m-1}\eps_n+X^{j,j}_{t_n}<\paren{m-1}\eps_n-MAt_n
<0
$ and $t_n\leq \eps_n/M$, it follows that $\tilde Z^j$ reaches zero before $\tilde C^j$ reaches $\eps_n/M$. 
Hence, $\tilde C^j\leq \eps_n/M$ and so $C=0$ on $[0,T]$. 
We have hence shown  that $\tau>0$. 

We now the following identity in law:
\begin{linenomath}
\begin{equation}
\label{identityInLawForShiftedProcessesAtFirstSpontaneousGenerationEquation}
\left.\paren{\imf{X^{1}}{C^1_\tau+\cdot}-\imf{X^{1}}{C^1_\tau}, \ldots, \imf{X^{m}}{C^m_\tau+\cdot}-\imf{X^{m}}{C^m_\tau}, Y_{\tau+t}-Y_\tau}
\right| \tau<\infty
\stackrel{d}{=}
\paren{X^1,\ldots, X^m, Y}.
\end{equation}%
\end{linenomath}
The identity in law \eqref{identityInLawForShiftedProcessesAtFirstSpontaneousGenerationEquation} implies a contradiction since if $\tilde Z$ satisfies \eqref{odeSystem} but driven by the left-hand side of \eqref{identityInLawForShiftedProcessesAtFirstSpontaneousGenerationEquation} with initial value $Z_\tau$ then $\tilde Z$ should, by definition of $\tau$, have spontaneous generation at time $0$, which is impossible.

To prove the identity in law of \eqref{identityInLawForShiftedProcessesAtFirstSpontaneousGenerationEquation}, 
we first prove that for any $0\leq t_1,\ldots ,t_m$ we have
\begin{linenomath}
\begin{equation}
\label{measurabilityDetailForFirstSpontaneousGeneration}
A=\set{C^1_{\tau}\leq t_1,\ldots, C^m_{\tau}\leq t_m, \tau> t}\in \F^\circ_{t_1,\ldots, t_m,t}. 
\end{equation}\end{linenomath}Indeed, the above will follow from proving that
\begin{linenomath}\begin{esn}
\set{\tau<t}\in 
\sag{C^j_s,Z^j_s,Y^j_s, X^{i,j}\circ C^i_s:0\leq s\leq t, 1\leq i,j\leq m}. 
\end{esn}\end{linenomath}thanks to Lemma \ref{measurabilityPropertiesLemma}.  However, note that $\tau<t$ is and only if there exists $J\subset[m]=\set{1,\ldots,m}$ and $j\in J$ such that $Z^j$ presents spontaneous generation over a common interval of constancy of $Y^j$ and $X^{i,j}\circ C^i$ of length greater than $\eps$. This can be discretized as follows:
\begin{linenomath}\begin{align*}
\set{\tau<t}
&=\bigcup_{J\subset [m]}\bigcup_{j\in J}\bigcup_{\substack{\eps>0\\ \eps\in \ra}}\bigcup_{\substack{q\in [0,t]\\q\in\ra}} \bigcap_{\substack{\delta>0\\ \delta\in\ra}}\bigcup_{\substack{p<q-\eps\\p\in \ra}}
\bigcap_{i\not\in J}\bigcap_{j'\in J}
\\&\set{Z^j_q>0,  0<\underline Z^j_{[p,q]}, Z^{j'}_{p}<\delta}\cap\set{X^{i,j}\circ C^i, Y^j\text{ are constant on } [p,q]},
\end{align*}\end{linenomath}where $\underline Z^j_{[p,q]}=\inf\set{ Z^j_s:s\in [p,q]}$. 

We now consider the random times $\tau_n$ and $C^i_n$ where
\begin{linenomath}
\begin{esn}
\tau_n=(k+1)/2^n\text{ if }k/2^n\leq \tau<(k+1)/2^n
\end{esn}\end{linenomath}and
\begin{linenomath}
\begin{esn}
C^i_n=(k_i+1)/2^n\text{ if }k_i/2^n\leq C^i_{\tau_n}<(k_i+1)/2^n. 
\end{esn}\end{linenomath}Then, thanks to \eqref{measurabilityDetailForFirstSpontaneousGeneration}
\begin{linenomath}
\begin{align*}
&\set{\tau=(k+1)/2^n, C^i_n=(k_i+1)/2^n}
\\&=\set{k/2^n\leq \tau<(k+1)/2^n, k_i/2^n\leq
C^i_{(k+1)/2^n}<(k_i+1)/2^n }
\\&\in 
\F^\circ_{(k_1+1)/2^n,\ldots, (k_m+1)/2^n,(k+1)/2^n}
\end{align*}\end{linenomath}Also, note that $\tau_n$ and $C^i_n$ decrease to $\tau$
and $C^i_\tau$ respectively. 

Consider now the processes $\tilde X^i$ and $\tilde Y$ where $\tilde
X^i_t=X^i_{C^i_n+t}-X^i_{C^i_n}$ and $\tilde
Y_{t}=Y_{\tau_n+t}-Y_{\tau_n}$. We assert now that the joint law of
$\tilde X^1,\ldots, \tilde X^m$ and $\tilde Y$ equals the law of
$X^1,\ldots, X^m$ and $Y$. To prove this, we focus on the
one-dimensional distributions since the computation of the finite-dimensional
distributions is just notationally more cumbersome. 
\begin{linenomath}
\begin{align*}
&\proba{\tilde X^i_{t}\leq x_i,1\leq i\leq m , \tilde
Y_t\leq x,\tau_n<\infty}
\\&=\sum_{k_1,\ldots, k_m,k}
\proba{ \tau_n=(k+1)/2^n, C^i_n=(k_i+1)/2^n, \tilde X^i_{t}\leq x_i, 1\leq i\leq m, \tilde
Y^t\leq x}
\\&=\sum_{k_1,\ldots, k_m,k}\p \left( k/2^n\leq \tau<(k+1)/2^n, k_i/2^n\leq
C^i_{(k+1)/2^n}<(k_i+1)/2^n ,\right.
\\& \hphantom{=\sum_{k_1,\ldots, k_m,k}\p (}\left. X^i_{t+(k_i+1)/2^n}-X^i_{(k_i+1)/2^n}\leq x ,1\leq i\leq m, 
Y_{t+(k+1)/2^n}-Y_{(k+1)/2^n}\leq x \right)
\\&=\sum_{k_1,\ldots, k_m,k}\proba{ k/2^n\leq \tau<(k+1)/2^n, k_i/2^n\leq
C^i_{(k+1)/2^n}<(k_i+1)/2^n,1\leq i\leq m}
\\& 
\hphantom{=\sum_{k_1,\ldots, k_m,k}}\times\proba{X^i_{t}\leq x_i,1\leq i\leq m, 
Y_{t}\leq x}
\\&= \proba{\tau_n<\infty}\proba{X^i_{t}\leq x_i,1\leq i\leq m, 
Y_{t}\leq x}.
\end{align*}\end{linenomath}

As $n\to\infty$, the process $\tilde X^i$  converges to $X^i_{C^i_\tau+\cdot}-X^i_{C^i_\tau}$. 
We conclude \eqref{identityInLawForShiftedProcessesAtFirstSpontaneousGenerationEquation}.
\end{proof}

\section{Construction of affine processes on $\re_+^m$}
\label{nonNegativeAffineProcessesSection}
In this section we aim at completing the proof of Theorem \ref{affineProcessRepresentationTheorem} in the case where the process takes values in $\re_+^m$; that is, when $n=0$. 

Let $X^1,\ldots, X^m, Y$ be L\'evy processes satisfying the conditions of Theorem \ref{affineProcessRepresentationTheorem} when $n=0$. 
Let $z\in\re_+^m$ have non-negative coordinates, let $Z$ be the unique solution to \eqref{affineProcessRepresentationEquation} (when there is uniqueness and let $Z$ be zero otherwise) and let $C$ be the (coordinatewise) primitive of $Z$ which starts at zero. 
(The solution exists thanks to Lemma \ref{generalExistenceLemma} and is unique by Lemmas \ref{monotonicityLemma} and \ref{noSpontaneousGenerationForLevyProcessesLemma}). 

For $t,t_1,\ldots,t_m\geq 0$, recall the definition of the multiparameter filtration $\F^\circ_{t_1,\ldots, t_m,t}$ given in \eqref{multiparameterFiltrationDefinition}. 
Let $\mc{N}$ be the null sets of $\p$ and define
\begin{linenomath}
\begin{equation}
\label{completedMultiparameterFiltrationDefinition}
\F_{t_1,\ldots, t_m,t}=\F^\circ_{t_1,\ldots,t_m,t}\vee \mc{N}.
\end{equation}
\end{linenomath}

Since the processes $X^1,\ldots, X^m,Y$ are independent and the completed filtrations of any one of the L\'evy processes are right-continuous (cf. \cite[Prop. 4, Ch. 1]{MR1406564}) then one can use \cite[Ex. 2.5]{MR2016344} to see that
\begin{linenomath}
\begin{esn}
\F_{s_1,\ldots, s_m,s}=\bigcap_{t_i>s_i,t>s}\F_{t_1, \ldots, t_m,t}. 
\end{esn}
\end{linenomath}

\begin{lemma}[Measurability details and the Markov property]
\label{measurabilityDetailsLemma}
\mbox{}
\begin{enumerate}
\item For any $t\geq 0$, $C_t$ is a multidimensional stopping time:
\begin{linenomath}
\begin{esn}
\set{C^1_t\leq t_1,\ldots, C^m_t\leq t_m}\in \F_{t_1,\ldots, t_m,t}. 
\end{esn}
\item The class\begin{equation}
\label{timeChangedFiltrationDefinition}
\G_t=\set{A\in\F: A\cap\set{C^1_{t}\leq t_1,\ldots, C^m_t\leq t_m}\in \F_{t_1,\ldots, t_m,t}}
\end{equation}is a \sa\ and the collection $(\G_t,t\geq 0)$ is a filtration satisfying the usual hypotheses. 
\item The following strong Markov property holds: for any $t\geq 0$, conditionally on $C^i_t<\infty$ for $1\leq i\leq m$
\begin{esn}
\paren{X^1_{C^1_t+\cdot}-X^1_{C^1_t},\ldots, X^{m}_{C^m_t+\cdot}-X^m_{C^m_t},Y_{t+\cdot}-Y_t}\stackrel{d}{=} \paren{X^1,\ldots, X^m,Y}
\end{esn}and the process on the left-hand side is independent of $\G_t$.
\item $Z$ is a a $\paren{\G_t,t\geq 0}$-Markov process. 
\end{linenomath}
\end{enumerate}
\end{lemma}
\begin{proof}
\begin{enumerate}
\item The fact that $C$ is a multidimensional stopping time follows from Lemma \ref{measurabilityPropertiesLemma} once we note that $C$ is almost surely equal to the maximal solution to \eqref{affineProcessRepresentationEquation} constructed in Lemma \ref{generalExistenceLemma} thanks to the fact that solutions have almost surely no spontaneous generation (Lemma \ref{noSpontaneousGenerationForLevyProcessesLemma}) and the uniqueness of solutions without spontaneous generation of Lemma \ref{monotonicityLemma}.
\item It is easy to prove that $\G_t$ is a \sa. 
$\G_t$ also contains the null sets $\mc{N}$ since every $\F_{t_1,\ldots, t_m, t}$ contains them by definition. 
Also, since $\F_{t_1,\ldots, t_m,s}\subset\F_{t_1.\ldots,t_m,t}$ if $s\leq t$, then $(\G_t,t\geq 0)$ is a filtration. 
To see that $\G_t$ is right continuous, we only need to prove that $\cap_{t>s}\G_t=\G_s$. 
Let $A\in \cap_{t>s}\G_t$. Then for $t'>t>s$ we have
\begin{linenomath}
\begin{esn}
A\cap\set{C^1_{t}\leq t_1,\ldots, C^m_{t}\leq t_m}\in\F_{t_1,\ldots, t_m,t'}. 
\end{esn}\end{linenomath}Since $C^i_t\to C^i_s$ as $t\downarrow s$, we see that
\begin{linenomath}\begin{esn}
A\cap\set{C^1_{t}\leq t_1,\ldots, C^m_{t}\leq t_m}\uparrow A\cap\set{C^1_{s}\leq t_1,\ldots, C^m_{s}\leq t_m}
\end{esn}\end{linenomath}as $t\downarrow s$ and we conclude that
\begin{linenomath}
\begin{esn}
A\cap\set{C^1_{s}\leq t_1,\ldots, C^m_{s}\leq t_m}\in\F_{t_1,\ldots, t_m,t'}
\end{esn}\end{linenomath}for $s<t'$. Finally, we have already remarked that $\F_{t_1,\ldots, t_m,t'}\downarrow \F_{t_1,,\ldots, t_m,s}$ as $t'\downarrow s$, proving that $A\in\G_s$. 
\item The proof of the Markov type property follows the same pattern as the one in Lemma \ref{noSpontaneousGenerationForLevyProcessesLemma}. In fact, it is basically the same proof as the strong Markov property for L\'evy processes at a stopping time. 
\item First, from Lemma \ref{measurabilityPropertiesLemma} we deduce the existence of a measurable map $F_t$ which applied to functions $(f^j,g^j)$ satisfying \defin{H} returns the value at $t$ of the (maximal) solution $h$ to \eqref{odeSystem} (which is the unique one when inputting L\'evy processes plus an initial value). Note also that $t\mapsto \tilde C^j_t=C^j_{t+s}-C^j_s$ has a \cadlag\ derivative $t\mapsto \tilde Z^j_t=Z^j_{t+s}$ and satisfies
\begin{linenomath}
\begin{esn}
\tilde Z^j_t= Z^j_{s}+\sum_i \tilde X^{i,j}_{\tilde C^j_s}+\tilde Y^j_t
\end{esn}\end{linenomath}where
\begin{linenomath}\begin{esn}
\tilde X^{i,j}_t=X^{i,j}_{t+C^i_s}-X^{i,j}_{C^i_s}\quad\text{and} \quad \tilde Y_t=Y_{t+s}-Y_s. 
\end{esn}\end{linenomath}Hence $Z_{t+s}=\imf{F_t}{\tilde X^1,\ldots, \tilde X^m, Z_s+\tilde Y}$. Since $\tilde X^1,\ldots, \tilde X^m,\tilde Y$ are independent of $Z^j_s, C^j_s$ (which are $\G_s$-measurable), we see that the conditional law of $Z_{t+s}$ given $\G_s$ equals the law of $Z$ started at $\tilde z$ on the set $Z_s=\tilde z$. \qedhere
\end{enumerate}
\end{proof}
We now consider a martingale which is fundamental to the proof of Theorem \ref{affineProcessRepresentationTheorem}. 
\begin{lemma}[An exponential martingale]
\label{exponentialMartingaleLemma}
For any $u\in\re_-^m$, the stochastic process $M$ given by
\begin{linenomath}
\begin{esn}
M_t=e^{u\cdot Z_t}-\int_0^t e^{u\cdot Z_s}\bra{\imf{F}{u}+\imf{R}{u}\cdot Z_s}\, ds
\end{esn}\end{linenomath}%
is a martingale. 
\end{lemma}
\begin{proof}
Since $M$ has bounded paths on $[0,t]$ for any $t\geq 0$, since $u\in\re^m_-$,it suffices to prove that $M$ is a local martingale. 
Consider the exponential martingale associated to any $X^i$ and to $Y$: 
since\begin{linenomath}\begin{esn}
\esp{e^{u\cdot X^i_1}}=e^{\imf{R_i}{u}}
\quad\text{and}\quad
\esp{e^{u\cdot Y_1}}=e^{\imf{F}{u}},
\end{esn}\end{linenomath}and since $x\mapsto e^{u\cdot x}$ is bounded on $E=\re_+^m\times\re^n$ if $u\in \re_-^m$, the stochastic processes
\begin{linenomath}\begin{esn}
\tilde M^i_t=e^{u\cdot X^i_t}-\int_0^t \imf{R_i}{u} e^{u \cdot X^i_s}\, ds
\quad\text{and}\quad 
N_t=e^{u\cdot Y_t}-\int_0^t \imf{F}{u} e^{u\cdot Y_s}\, ds
\end{esn}\end{linenomath}are martingales. 
(Note that the above assertion is true even if $C^i_t=\infty$ for some $t$ and $i$.) 
They are independent since  $X^1,\ldots, X^m$ and $Y$ are independent. 

The random variables $\paren{C_t,t\geq 0}$ are $\F_{s_1,\ldots, s_m,t}$-stopping times and $s\leq t$ implies $C_s\leq C_t$. 
Hence, they constitute a multiparameter time change in the sense of Chapter 6 of \cite{MR838085}. 
Consider then the time changed processes\begin{linenomath}\begin{esn}
M^i_t=\tilde M^i\circ C^i_t=e^{u\cdot X^i\circ C^i_t}-\int_0^t  e^{u\cdot X^i\circ C^i_s} \imf{R_i}{u} Z^i_s\, ds. 
\end{esn}\end{linenomath}Problem 19 in \cite[Ch. 2]{MR838085} tells us that the (multiparameter) time change of the $m+1$ independent martingales $\tilde M^1,\ldots, \tilde M^m, N$ gives rise to the $m+1$ orthogonal local martingales $M^1,\ldots, M^n,N$. 
Hence\begin{linenomath}\begin{esn}
\bra{M^i,M^j}=0=
\bra{M^i,N}
\end{esn}\end{linenomath}for all $i,j$ with $i\neq j$. 

Note that\begin{linenomath}\begin{align*}
e^{u\cdot Z_t}
=\prod_{j}e^{u_j\cdot Z^j_t }
=\prod_j e^{u_j \cdot Y^j_t}\prod_i e^{u_j \cdot X^{i,j}\circ C^i_t}
=e^{u\cdot Y_t}\prod_i e^{u\cdot X^i\circ C^i_t}.
\end{align*}\end{linenomath}Since $e^{u \cdot X^i\circ C^i}$ and $e^{u\cdot Y}$ are semimartingales whose local martingale parts are orthogonal, and whose finite variation parts are continuous, we can use integration by parts, the fact that covariation is bilinear and that the covariation with a continuous finite-variation process is zero (cf. Theorem 26.6.viii in \cite{MR1876169}) to obtain
\begin{linenomath}\begin{align*}
e^{u\cdot Z_t}
&=\text{Loc. Mart.}+
\sum_{j}\int_0^t e^{u\cdot Y_s}\prod_{i\neq j} e^{u\cdot X^i\circ C^j}e^{u\cdot Z^j_s}\imf{R_j}{u} Z^j_s\, ds
+\int_0^t \prod_i e^{u\cdot X^i\circ C^i_t}e^{u\cdot Y_s}\imf{F}{u}\, du
\\&=\text{Loc. Mart.}+\int_0^t e^{u\cdot Z_s}\bra{\imf{R}{u}\cdot Z_s+\imf{F}{u}}\, ds. 
\end{align*}\end{linenomath}We conclude that $M$ is a local martingale. 
\end{proof}
We deduce the following result, which is important in our proof of stability of the multiparameter time change equation. 
Indeed, it is important since addition is not continuous on the space of \cadlag\ functions (with the Skorohod $J_1$ topology), but it is continuous when the summands do not have common discontinuities, as is discussed for example in Theorem 4.1 of \cite{MR561155}. 
\begin{corollary}
\label{noCommonJumpsCorollary}
Almost surely, for each $j\in\set{1,\ldots, m}$ the processes   $X^{i,j}\circ C^i, 1\leq i\leq m$ and $Y^j$ do not jump at the same time. 
\end{corollary}
\begin{proof}
As shown in the proof of Lemma \ref{exponentialMartingaleLemma},  the processes $e^{ u\cdot X^i\circ C^i}, 1\leq i\leq m$ and $e^{u\cdot Y}$ are semimartingales for $1\leq i\leq m$ with zero covariation. 
Considering a vector $u$ all of whose coordinates are zero except the $j$-th which equals $-1$, we deduce that the semimartingales $e^{-X^{i,j}\circ C^i}, 1\leq i\leq m$ and $e^{-Y^j}$ have zero covariation. 
Since $e^{-Y^j}$ is of finite variation, we see that
\begin{lesn}
0=[e^{-X^{i,j}\circ C^i}, e^{-Y^j}]_t=\sum_{s\leq t } \Delta e^{-X^{i,j}\circ C^i(s)}\Delta e^{-Y^j_s}.
\end{lesn}Since each summand in the right-most side is negative, we conclude that $X^{i,j}\circ C^i$ and $Y^j$ do not jump at the same time. 
The same argument applies when considering $X^{i,j}\circ C^i$ and $X^{i',j}\circ C^{i'}$ if $i\neq i'$ since at most one is of infinite variation. 
\end{proof}

As already mentioned in Section \ref{preliminariesSection}, there exists a unique function $\imf{\psi}{t,u}$ such that $\imf{\psi}{0,u}=u$ and
\begin{linenomath}\begin{esn}
\frac{\partial \imf{\psi}{t,u}}{\partial t}=\imf{R\circ \psi}{t,u}. 
\end{esn}\end{linenomath}We also consider the function\begin{linenomath}\begin{esn}
\imf{\phi}{t,u}=\int_0^t \imf{F\circ \psi}{s,u}\, ds. 
\end{esn}\end{linenomath}
In order to prove that the process $Z$ which solves \eqref{affineProcessRepresentationEquation} when $n=0$ is a $\cbi$ associated to the pair of characteristic exponents $R$ and $F$, it suffices to see that $Z$ is a Markov process (which is covered by Lemma \ref{measurabilityDetailsLemma}), and to prove the following lemma:
\begin{lemma}
\label{oneDimensionalDistributionsLemma}
For any $z\in\re_+^m$,  and $u\in\re_-^m$\begin{linenomath}\begin{esn}
\imf{\se_z}{e^{ u\cdot Z_t}}= e^{z\cdot \imf{\psi}{t,u}+\imf{\phi}{t,u}}. 
\end{esn}\end{linenomath}
\end{lemma}
\begin{proof}
Let\begin{linenomath}\begin{esn}
\imf{G}{s}=\imf{\se_z}{e^{\imf{\psi}{t-s,u}\cdot Z_s+\imf{\phi}{t-s,u}}}
\end{esn}\end{linenomath}for $s\in [0,t]$. 
We will show that $\imf{G'}{s}=0$ for any $s\in (0,t)$ which implies that\begin{linenomath}\begin{esn}
e^{z\cdot \imf{\psi}{t,u}+\imf{\phi}{t,u}}=\imf{G}{0}=\imf{G}{t}=\imf{\se_z}{e^{u \cdot Z_t}}
\end{esn}\end{linenomath}and hence finishes the proof.

To this end, write\begin{linenomath}\begin{align*}
\imf{G}{s+h}-\imf{G}{s}
&=\imf{\se_z}{e^{\imf{\psi}{t-s-h,u}\cdot Z_{s+h}+\imf{\phi}{t-s-h,u}}  - e^{\imf{\psi}{t-s-h,u}\cdot Z_{s}+\imf{\phi}{t-s-h,u}}}
\\&+\imf{\se_z}{e^{\imf{\psi}{t-s-h,u}\cdot Z_{s}+\imf{\phi}{t-s-h,u}}  - e^{\imf{\psi}{t-s,u}\cdot Z_{s}+\imf{\phi}{t-s,u}}} 
\end{align*}\end{linenomath}Taking expectations in  Lemma \eqref{exponentialMartingaleLemma}, we see that\begin{linenomath}\begin{align*}
&\frac{1}{h}\imf{\se_z}{e^{\imf{\psi}{t-s-h,u}\cdot Z_{s+h}+\imf{\phi}{t-s-h,u}}  - e^{\imf{\psi}{t-s-h,u}\cdot Z_{s}+\imf{\phi}{t-s-h,u}}} 
\\&=\frac{1}{h}
\int_{s}^{s+h}\imf{\se_z}{ e^{\imf{\psi}{t-s-h,u}\cdot Z_{r}+\imf{\phi}{t-s-h,u}}  \bra{  \imf{F\circ \psi}{t-s-h,u}+\imf{R\circ \psi}{t-s-h,u}\cdot Z_r}}\, dr
\\&\to \imf{\se_z}{ e^{\imf{\psi}{t-s,u}\cdot Z_{s}+\imf{\phi}{t-s,u}}  \bra{  \imf{F\circ \psi}{t-s,u}+\imf{R\circ \psi}{t-s,u}\cdot Z_s}}.
\end{align*}\end{linenomath}%
On the other hand, we can differentiate under the expectation to obtain:
\begin{linenomath}\begin{align*}
&\frac{1}{h}\imf{\se_z}{e^{\imf{\psi}{t-s-h,u}\cdot Z_{s}+\imf{\phi}{t-s-h,u}}  - e^{\imf{\psi}{t-s,u}\cdot Z_{s}+\imf{\phi}{t-s,u}}}
\\&\to -\imf{\se_z}{e^{\imf{\psi}{t-s,u}\cdot Z_{s}+\imf{\phi}{t-s,u}}\bra{\imf{F\circ\psi}{t-s,u}\cdot Z_s+\imf{R\circ \phi}{t-s,u} } }.
\end{align*}\end{linenomath}We conclude that $\imf{G'}{s}=0$. 
\end{proof}

\begin{proof}[Summary and conclusion of the proof of Theorem \ref{affineProcessRepresentationTheorem} when $n=0$]
Existence for solutions to \eqref{affineProcessRepresentationEquation} are covered by Lemma \ref{generalExistenceLemma} and is valid more generally. 
Uniqueness is then covered, through the concept of spontaneous generation, in Lemma \ref{noSpontaneousGenerationForLevyProcessesLemma}. 
Lemma \ref{measurabilityDetailsLemma} then proves that the unique solution to \eqref{affineProcessRepresentationEquation} is a Markov process and thanks to Lemma \ref{oneDimensionalDistributionsLemma} we can identify its one-dimensional distributions with those of a $\cbi$ process associated to the parameters of the underlying L\'evy processes called $R$ and $F$. 
\end{proof}

\section{Construction of affine processes on $\re_+^m\times\re^n$}
\label{generalConstructionSection}
Let $X^1,\ldots, X^m$, $Y$ be L\'evy processes satisfying the conditions of Theorem \ref{affineProcessRepresentationTheorem}. 
Let $R^i$ and $F$ be the characteristic exponents of $X^i$ and $Y$ as in Equation \eqref{characteristicExponentsDefinition} and let $R=\paren{R^1,\ldots, R^{n+m}}$, where we set $R^i=0$ for $m+1\leq i\leq m+n$. 
With the first $m$ coordinates of these processes we solve \eqref{affineProcessRepresentationEquation} to obtain the non-negative processes $Z^1,\ldots, Z^m$ analyzed in Section \ref{nonNegativeAffineProcessesSection}. 
We can then (re)define $Z$ by setting $Z^{m+j}= z+\sum_{i=1}^m X^{i,j}\circ C^i+Y^j$ and note that $Z$ solves Equation \eqref{affineProcessRepresentationEquation} when $\beta=0$. 
In this case, we can follow the arguments of the case $n=0$ presented in Section \ref{nonNegativeAffineProcessesSection}  to see that $Z$ is a Markov process and that its one-dimensional distributions are characterized by the computation\begin{lesn}
\imf{\se_z}{e^{ u\cdot Z_t}}=e^{z\cdot \imf{\psi}{t,u}+\imf{\phi}{t,u}},
\end{lesn}valid for $u\in\re_-^m\times i\re^n$, where $\psi$ and $\phi$ solve the Riccati equations\begin{lesn}
\frac{\partial }{\partial t}\imf{\psi}{t,u}=\imf{R\circ \psi}{t,u}
\quad\text{and}\quad
\frac{\partial }{\partial t}\imf{\phi}{t,u}=\imf{F\circ \psi}{t,u}
\end{lesn}with initial conditions $\imf{\psi}{0,u}=u$ and $\imf{\phi}{0,u}=0$. 
This proves Theorem \ref{affineProcessRepresentationTheorem} when $\beta=0$. 
Affine processes of this type have been dubbed partially additive in \cite{MR2851694} since the law of $\tilde z+Z$ under $\p_{z}$ equals $\p_{\tilde z+z}$ whenever $\tilde z$ has its first $m$ components equal to zero.  

We now extend the process $Z=\paren{Z^1,\ldots, Z^{m+n}}$ just considered to obtain the full proof of Theorem \ref{affineProcessRepresentationTheorem}. 
To do this, consider the equations\begin{linenomath}\begin{esn}
Z^{\beta,m+j}_t=z_{m+j}+\sum_{i=1}^m X^{i,m+j}\circ C^i_t+Y^{m+j}_t+\sum_{i=1}^{n}C^{\beta,i+m}_t\beta_{i,j} \quad  1\leq j\leq n
\end{esn}\end{linenomath}where $C^{\beta, j}_t=\int_0^t Z^{\beta,j}_s\, ds$ for $m+1\leq j\leq m+n$ (and $Z^{\beta,j}=Z^j$ for $1\leq j\leq m$). 
If we let extend the matrix $\beta$ to be $(m+n)\times (m+n)$ by adding zeros at coordinates $i,j$ if $i\leq m$ or $j\leq m$, the equations become:
\begin{linenomath}
\begin{equation}
\label{OUODE}
Z^\beta=Z+C^\beta \beta. 
\end{equation}
\end{linenomath}This is a linear stochastic differential equation driven by  $Z$ which, of course, admits an unique solution. 
This is for example contained in \cite[Ch. 9\S V]{MR2020294}, where the following explicit formula is given:
\begin{linenomath}
\begin{esn}
 Z^\beta_t=e^{t\beta }z+\int_0^t e^{\paren{t-s}\beta}\, dZ_s. 
\end{esn}\end{linenomath}

We first construct an exponential martingale, which takes the place of Lemma \ref{exponentialMartingaleLemma}. 
\begin{lemma}
For any $u\in\re_+^m\times i\re^n$, the stochastic process $M^\beta$ given by\begin{lesn}
M^\beta_t=e^{u\cdot Z^\beta_t}-\int_0^t e^{u\cdot Z^\beta_s}\bra{\imf{F}{u}+\paren{\imf{R}{u}+\beta u}\cdot Z^\beta_s}\, ds
\end{lesn}is a martingale. 
\end{lemma}
\begin{proof}
Lemma \ref{exponentialMartingaleLemma} can be extended to the case $\beta=0$, proving that $M^0$ is a martingale. 
Now, note that\begin{lesn}
e^{u\cdot Z^\beta_t}=e^{u\cdot Z_t}e^{C^\beta_t\beta u}. 
\end{lesn}We now apply integration by parts, noting that since $C^\beta \beta u$ is continuous and of finite variation, then $\bra{e^{u\cdot Z}, C^\beta\beta u}=0$. 
We then obtain\begin{linenomath}
\begin{align*}
e^{u\cdot Z^\beta_t}
&=e^{u\cdot z}+\int_0^t e^{C^\beta_s\beta u}\, dM^0_s +\int_0^t  e^{C^\beta_s\beta u} \bra{Z_s\cdot \imf{R}{u}+\imf{F}{u}} e^{u\cdot Z_s}\, ds+\int_0^t e^{u\cdot Z_s}Z_s\beta u e^{C^\beta_s\beta u}\, ds. 
\\&=\text{Loc. Mart.} +\int_0^t e^{u\cdot Z^\beta_s}\bra{\imf{F}{u}+\bra{\imf{R}{u}+\beta u}\cdot Z^\beta_s}\, ds.\qedhere
\end{align*}
\end{linenomath}
\end{proof}Adapting the proof of Lemma \ref{oneDimensionalDistributionsLemma}, we see that\begin{lesn}
\imf{\se_z}{e^{u\cdot Z^\beta_t}}=e^{z\cdot \imf{\psi^\beta}{t,u}+\imf{\phi^\beta}{t,u}},
\end{lesn}where $\psi^\beta$ and $\phi^\beta$ satisfy the Riccati equations\begin{lesn}
\frac{\partial }{\partial t}\imf{\psi^\beta}{t,u}=\imf{R\circ \psi^\beta}{t,u}+\beta\imf{\psi^\beta}{t,u}
\quad\text{and}\quad
\frac{\partial }{\partial t}\imf{\phi^\beta}{t,u}=\imf{F\circ \psi^\beta}{t,u}
\end{lesn}with initial conditions $\imf{\psi^\beta}{0,u}=u$ and $\imf{\phi^\beta}{0,u}=0$. 

We now finish the proof that $Z^\beta$ is an affine process, thereby proving Theorem \ref{affineProcessRepresentationTheorem} in the remaining case when $n\neq 0$. 
Since we have already determined the one-dimensional distributions of $Z^\beta$, it remains to discuss the Markov property. 
Note that\begin{lesn}
Z^\beta_{t+s}=Z^\beta_t+Z_{t+s}-Z_t+\beta\int_t^{s+t} Z^\beta_r\, dr. 
\end{lesn}
Therefore, $Z^\beta_{t+\cdot}$ satisfies the same differential equation as $Z^\beta$ but starting at $Z^\beta_t$ and driven by $Z_{s+t}- Z_t$. 
Recall that the first $m$ coordinates of $Z^\beta$ equal those of $Z$. 
Let $\paren{\G_t,t\geq 0}$ be the filtration defined in Section \ref{nonNegativeAffineProcessesSection} and with respect to which the Markov property of $Z$ holds. 
Since $Z$ is partially additive, then \begin{lesn}
\text{the law of $Z^1_{t+\cdot},\ldots, Z^m_{t+\cdot}, Z^{m+1}_{t+\cdot}-Z^{m+1}_{t},\ldots, Z^{m+n}_{t+\cdot}-Z^{m+n}_{t}$ given $\G_t$ is $\p_{Z^1_t,\ldots, Z^m_t, 0\,\ldots, 0}$.}
\end{lesn} 
This shows that the law of $Z^\beta_{t+\cdot}$ given $\G_t$ equals the law of $Z^\beta$ under $\p_{Z^\beta_t}$ which proves that $Z^\beta$  is an affine Markov process whose infinitesimal parameters we have already determined. 

\section{Stability analysis of the time change transformation: approximation and limit theorems}
\label{stabilityAnalysisSection}
In this section, we will give a stability analysis related to the stochastic system \eqref{affineProcessRepresentationEquation} through the deterministic system \eqref{odeSystem}, aiming at a proof of Theorem \ref{stochasticLimitTheorem}. 
For the stability analysis we need to consider not only the system \eqref{affineProcessRepresentationEquation} but a differential inequality that turns up naturally. 
This differential inequality is analyzed in Subsection \ref{differentialInequalitySubsection}. 
The stability analysis is then performed in Subsection \ref{stabilityAnalysisSubsection}
which enables us to obtain some applications to approximations and limit theorems concerning affine processes in Subsection \ref{applicationsOfStabilitySubsection}. 
\subsection{A differential inequality}
\label{differentialInequalitySubsection}
Recall the setting of Theorems \ref{affineProcessRepresentationTheorem} and \ref{stochasticLimitTheorem}. 
If $C^l$ converges to $C$, it might happen that $C^{i,l}_t\leq C^i_t$ or $C^{i,l}_t\geq C^i_t$. Hence, we can only infer that\begin{lesn}
X^{i,j}_-\circ C^i_t \leq \liminf_l X^{i,j,l}\circ C^{i,l}_t\leq \limsup_l X^{i,j,l}\circ C^{i,l}_t\leq X^{i,j}\circ C^i_t.
\end{lesn}The following proposition 
is useful in determining 
whether the limit $X^{i,j,l}\circ C^{i,l}$ exists for most values of $t$. 
\begin{proposition}
\label{differentialInequalityProposition}
Under the setting of Theorem \ref{affineProcessRepresentationTheorem}, 
the associated cumulative population $C=(C^1,\ldots, C^m)$ is the unique non-decreasing and continuous process 
satisfying the differential inequalities
\begin{equation}
\label{differentialInequalityForLevyProcesses}
\int_r^t \sum_i X^{i,j}_{-}\circ  C_{s}+Y^j_{s}\, ds\leq C^j_t-C^j_r\leq \int_r^t \sum_i X^{i,j}\circ { C_s}+Y^j_{s}\, ds. 
\end{equation}
\end{proposition}

As a preliminary result, let us see that $C$ itself satisfies both sides of the inequality. 
\begin{lemma}
\label{continuityAtConstancyIntervalsLemma}
Almost surely, if $t$ is such that $C^i$ is constant on an interval to the right of $t$ then $X^{i}$ is continuous at $C^i_t$. 
Hence, almost surely, for all $t>0$:
\begin{lesn}
C^j_t-C^j_r=\int_r^t \sum_{i}X^{i,j}_-\circ C^i_s+Y^j_s\, ds. 
\end{lesn}
\end{lemma}
\begin{proof}
The first statement of the lemma is obviously true at $t=0$. 
To handle every $t>0$ simultaneously, it suffices to prove that for any rational $q$, if $T^j_{q,0+}$ denotes the first zero of $Z^j$ after $q$ or zero, depending on if $Z^j_q>0$ or not, then $X^i$ is continuous at $C^i(T^j_{q,0+})$. 
This is basically a result of quasi-continuity of the L\'evy processes involved. 

If $T$ is any stopping time with respect to the filtration $(\G_t,t\geq 0)$ defined in \eqref{timeChangedFiltrationDefinition} then $\paren{C^1_{T},\ldots, C^m_T, T}$ is a stopping time with respect to the multiparameter filtration $\F_{t_1,\ldots, t_m,t}$ defined in \eqref{completedMultiparameterFiltrationDefinition}. 
This follows simply when  $T$ takes values in a discrete set $\set{a_k:k\in\na}$ because, by definition of $\G_{a_k}$, we see that\begin{lesn}
\set{C^1_T\leq t_1,\cdots, C^m_T\leq t_m, T=a_k}=\set{C^1_{a_k}\leq t_1,\cdots, C^m_{a_k}\leq t_m}\cap\set{ T=a_k}\in\F_{t_1,\ldots, t_m,a_k}. 
\end{lesn}When $T$ is a general stopping time, we approximate it by the decreasing sequence of stopping times $T_n$ which takes the value $k/2^n$ if $T\in [(k-1)/2^n,k/2^n)$. 

Let $T$ equal one of the $T_{q,0+}^j$ and note that $T$ is the increasing limit of $T_n$ where $T_n$ is the first time after $q$ that $Z^j$ is below $1/n$ or zero depending on if $Z^j_q>1/n$ or not. 
We always have $T_n\leq T$. 
If $Z^j_q>0$ then $T_n<T$ for all $n$. Recall that  $C^i_{T_n}$ is a stopping time for the filtration $\sigma (X^{i'}_{s'}, X^i_s: s'\geq 0, s\leq t, i'\neq i),t\geq 0$ defined for each $i$. 
Since $X^i$ is a L\'evy process with respect to that filtration, by quasi-continuity, we see that $X^i$ is continuous at $C^i_{T}$. 
\end{proof}
\begin{proof}[Proof of Proposition \ref{differentialInequalityProposition}]
Denote by $\tilde C$ any process satisfying the inequality \eqref{differentialInequalityForLevyProcesses}. 
Recall that, from Lemma \ref{generalExistenceLemma}, $C$ is obtained as the limit of $C^l$, where $C^l$ solves Equation \ref{affineProcessRepresentationEquation} driven by processes strictly bigger than $X^{i,j}$ and $Y^j$. 
The simple argument presented in Lemma \ref{firstMonotonicityLemma} implies that $\tilde C$ is bounded above by $C^l$ and therefore $\tilde C\leq C$. 
We now let $\tau=\inf\set{t\geq 0: \tilde C_t<C_t}$. 
By continuity, we see that $\tilde C= C$ on $[0,\tau]$. 
Let us suppose that $\tau<\infty$ to reach a contradiction. 
If $\tau<\infty$, there exist $1\leq i,j\leq m$ and $\eps_0>0$ such that for $0<\eps<\eps_0$ we have
\begin{linenomath}
\begin{equation}
\label{problematicInequalityForDifferentialInequality}
\int_\tau^{\tau+\eps} X^{i,j}_-\circ \tilde C^i_s\, ds<\int_\tau^{\tau+\eps} X^{i,j}\circ \tilde C^i_s\, ds. 
\end{equation}\end{linenomath}When $Z^i_{\tau-}>0$ then $C^i$  is strictly increasing to the right of $\tau$, implying that $\tilde C^i$ is strictly increasing to the right of $\tau$ and so \eqref{problematicInequalityForDifferentialInequality} does not hold. 
When $Z^i_{\tau-}=0$, $\tau$ cannot belong to the interior (or be the beginning) of an interval of constancy of $C^i$. Indeed, Lemma \ref{continuityAtConstancyIntervalsLemma} would then imply that $X^{i,j}_-\circ C^i_\tau=X^{i,j}\circ  C^i_\tau$ and that $C^i$ is constant to the right of $\tau$ which would contradict \eqref{problematicInequalityForDifferentialInequality}. 
Hence, $C^i$ increases on any right neighbourhood of $\tau$. 
However, recall that $C$ has no spontaneous generation. 
This implies the existence distinct indices $i_0,\ldots, i_k$ in $\set{1,\ldots, m}$ such that $Z^{i_0}_\tau>0$, $i_k=i$ and $X^{i_{l-1},i_l}\circ C^{i_{l-1}}$ is strictly increasing on a right neighbourhood of $\tau$. 
Starting with $i_0$ and using the fact that $C=\tilde C$ on $[0,\tau]$, then $X^{i_{l-1},i_{l}}\circ \tilde C^{i_{l-1}}$ is strictly increasing on a right neighbourhood of $\tau$ for every $l$ and hence \eqref{problematicInequalityForDifferentialInequality} can not hold either when $Z^i_{\tau-}=0$. 
\end{proof}

\subsection{Stability analysis}
\label{stabilityAnalysisSubsection}

The following result deals with stability of 
the multiparameter time changes of equation \eqref{affineProcessRepresentationEquation} in the deterministic setting of Section \ref{deterministicPathwiseAnalysisSection}. 
We focus on the case $n=0$ since our arguments can then handle the non-negative coordinates. 
Hence, we  will concern ourselves with equation \eqref{odeSystem} not only under changes
in $f^{i,j}$ and $g^i$ for $i,j=1,\dots,m$, but also with respect to
discretization of the transformation itself. 
Consider the following
approximation procedure: given $\sigma>0$, called the span, consider
the partition
\begin{lesn}
t_k=k\sigma,\quad k=0,1,2,\dots,
\end{lesn}
and construct a function
$c^{\sigma}=(c^{\sigma}_1,\dots,c^{\sigma}_m)$ by the recursion
\begin{lesn}
c^{\sigma}_j(0)=0\quad \text{for }j=1,\dots,m.
\end{lesn}
and for $t\in[t_{k-1},t_k)$:
\begin{linenomath}
\begin{equation}
c^{\sigma}_j(t)=c^{\sigma}_j(t_{k-1})+(t-t_{k-1})[\sum_{i=1}^m f^{i,j}\circ
c^{\sigma}_i(t_{k-1})+g^j(t_{k-1})]^+.
\end{equation}\end{linenomath}Equivalently, the function $c^{\sigma}$ is the unique solution to
the system of equations
\begin{lesn}
c^{\sigma}_j(t)=\int_{0}^t[\sum_{i=1}^m f^{i,j}\circ
c^{\sigma}_i(\floor{s/\sigma}\sigma)+g^j(\floor{s/\sigma}\sigma)]^+ds\quad\text{for
 }j=1,\dots,m.
\end{lesn}
The stability result is stated in terms of the usual Skorohod $J_1$
topology for c\'adl\'ag functions: a sequence $f_l$ converges to $f$
if each coordinate converges in the usual Skorohod $J_1$ topology.
This means that for each coordinate $f^j_l,l\geq 1$ there exist a sequence of
homeomorphisms $\lambda_l^j,l\geq 1$ of $[0,\infty)$ into itself such that
\begin{lesn}
f^j_l-f^j\circ\lambda_l^j\text{ and }\lambda_l^j-Id\text{ converge to 0  uniformly on
compact sets}.
\end{lesn}We will also use the uniform $J_1$ topology and which is
characterized by: a sequence of $f_l$ converges to $f$ if for  $1\leq j\leq m$  there
exists a sequence of homeomorphisms $\lambda_l^j$of $[0,\infty)$ into itself such
that
\begin{lesn}
f^j_l-f^j\circ\lambda_l^j\text{ and }\lambda_l^j-Id\text{ converges to 0 uniformly on
$[0,\infty)$}.
\end{lesn}

\begin{theorem}\label{te}
Let $(f^{i,j})_{i,j=1}^m$ and $g^j$ be c\'adl\'ag functions which satisfy
hypothesis \defin{H}, and suppose that there exists a unique non-decreasing $c$ which satisfies
\begin{linenomath}
\begin{equation}\label{d1}
\int_s^t \sum_{i=1}^m f^{i,j}_-\circ c^i(r)+g^j(r)\, dr\leq
c^j(t)-c^j(s)\leq \int_s^t \sum_{i=1}^m f^{i,j}\circ
c^i(r)+g^j(r)\, dr
\end{equation}\end{linenomath}for $s\leq t$, and $j=1,\dots,m $. 
(In particular, $c$ solves \eqref{odeSystem} and has a right-hand derivative $h$.) 
Let $\tau$ be the explosion time of $c$ defined by 
\begin{lesn}
\tau=\inf\set{t\geq 0: \exists j\text{ such that } c^j(t)=\infty}. 
\end{lesn}If $f_l^{i,j}\to f^{i,j}$ for $i,j=1,\dots,m$, and $g^j_l\to g^j$ in the Skorohod
$J_1$ topology, $\sigma_l\to 0$, and $c_l$ is any solution to
\begin{lesn}
c_l^j(t)=\int_0^t\left[\sum_{i=1}^m f_l^{i,j}\circ
c_l^i([s/\sigma_l]\sigma_l)+g_l^j([s/\sigma_l]\sigma_l)\right]^+ds
\end{lesn}then $c_l\to c$ pointwise and uniformly on compact sets $[0,\tau)$. 
Furthermore, if $f^{i,\cdot}\circ c^i$ and $f^{j,\cdot}\circ c^j$ do not jump at the
same time for $i\neq j$ and  $f^{i,\cdot }\circ c^i$ and $g$ do not
jump at the same time then the right-hand derivatives $D_+c_l$ converge to $h$
\begin{enumerate}
\item in the Skorohod $J_1$ topology if $\tau=\infty$.
\item in the uniform $J_1$ topology if $\tau<\infty$ and we
additionally assume that $f^{i,j}_l\to f^{i,j}$ for $i=1,\dots,m$ in the
uniform $J_1$ topology.
\end{enumerate}
\end{theorem}

Theorem \ref{stochasticLimitTheorem} follows from Theorem \ref{te} thanks to Lemma \ref{differentialInequalityProposition} and Corollary \ref{noCommonJumpsCorollary}. 

In order to prove  Theorem \ref{te} we will first prove a series of
lemmata.
\begin{lemma}
Under the assumptions of Theorem \ref{te}, if $(c_l(t),l\geq1)$ is
bounded for some $t>0$ then $c_l\to c$ uniformly on $[0,t]$.
\end{lemma}
\begin{proof}
Let $M$ be a bound for $c_l$ and let $K$ be an upper bound for
$(f^{i,\cdot}_l, l\geq1, i=1,\dots,m)$ on $[0,M]$ and $(g_l, l\geq 1)$ on
$[0,t]$ (which exists since $f^{i,j}_l\to f^{i,j}$ for $i,j=1,\dots,m$, and
$g_m\to g$). For any $s\in[0,t]$ we have that
\begin{equation}\label{co}
D_+c^j_l(s)=[\sum_{i=1}^m f^{i,j}_l\circ
c^i_l([s/\sigma_l]\sigma_l)+g_l^j([s/\sigma_l]\sigma_l)]^+\leq
(m+1)K\quad\text{ for any }i=1,\dots,m\, , 
\end{equation}
implying that the family of functions $\{c^j_l:l\geq 1\}$ has
uniformly bounded right-hand derivatives (on $[0,t]$) and starting
points. (If $\sigma_l=0$ we get the same upper bound for $D_+c_l^j$
using the equality $D_+c_l^j=\sum_{i=1}^m f_l^{i,j}\circ c^i_l
+g^j_l$.) Therefore $\{c^j_l:l\geq 1, j=1,\dots,m\}$ is uniformly
bounded and equicontinuous on $[0,t]$. This in
turn implies the same for the family $\{c_l:l\geq 1\}\subset
C([0,T],\mathbb{R}^m)$ on $[0,t]$. By the Arzel\`a-Ascoli theorem,
every sequence $(c_{l_k}, k\geq1)$ has a further subsequence that
converges to a function $\tilde{c}$ (which depends on the
subsequence). We now prove that $c=\tilde{c}$, which implies that
$c_l\to c$ as $l\to\infty$ uniformly on $[0,t]$.

Suppose that $l_k$ is such that $c_{l_k}$ has a limit
$\tilde{c}$ as $k\to\infty$ uniformly on $[0,t]$. Since each
$f^{i,j}$ has no negative jumps, we get
\begin{linenomath}
\begin{align*}
\liminf_{x\to y}f^{i,j}=f_-^{i,j}(y)
\quad\text{and}\quad
\limsup_{x\to y}f^{i,j}(x)=f^{i,j}(y)
\end{align*}
\end{linenomath}so that
\begin{linenomath}
\begin{align*}
f^{i,j}_{-}\circ \tilde{c}^i\leq\liminf_{k\to \infty}f_{l_k}^{i,j}\circ
\tilde{c}^i_{l_k}
\quad\text{ and}\quad
\limsup f_{l_k}^{i,j}\circ
\tilde{c}^i_{l_k}\leq f^{i,j}\circ \tilde{c}^i.
\end{align*}
\end{linenomath}
Using Fatou's lemma we get
\begin{linenomath}
\begin{align*}
\int_s^t[\sum_{i=1}^m f^{i,j}_-\circ \tilde{c}^i(r)+g^j_-(r)]^+dr
&\leq
\tilde{c}^j(t)-\tilde{c}^j(s)
\\&\leq \int_s^t[\sum_{i=1}^m f^{i,j}\circ
\tilde{c}^i(r)+g^j(r)]^+dr
\end{align*}\end{linenomath}for each $j=1,\dots,m$.
\end{proof}
\begin{proof}[Proof of Theorem \ref{te}]
For the convergence of the $c_l$ to $c$, we will argue along
sequences $l_k\to\infty$, considering the
following two cases: $(c_{l_k}(t))$ is bounded or goes to $\infty$. The
former alternative is handled by the previous theorem, for the
latter we prove that $c_{l_k}\to c$ pointwise on $[0,t]$ as
$k\to\infty$. The conclusion is that $c_l\to c$ pointwise on
$[0,\infty)$ and hence, by the previous lemma, uniformly on compact
sets of $[0,\tau)$.

Suppose that $l_k\to \infty$ is such that $\|c_{l_k}\|_{[0,t]} \to \infty$.
For any $x>0$, consider the sequence $c_{l_k}\wedge
x=(c^1_{l_k}\wedge x,\dots,c^m_{l_k}\wedge x)$. Note that it is
uniformly bounded.
Let $K$ be a common bound for $f^{i,j}_{k_l}$ on $[0,x]$ and for $g^j$ on $[0,t]$. 
For any $s\in[0,t]$
\begin{lesn}
\imf{D_+ (c^j_{l_k}\wedge x)}{s}=[ \sum_{i=1}^m f^{i,j}_{l_k}\circ
c^i_{l_k}([s/\sigma_{l_k}]\sigma_{l_k})+g^j_{l_k}([s/\sigma_{l_k}]\sigma_{l_k})]^+1_{\{c_{l_k}^j(s)\leq
x\}}\leq (m+1)K,
\end{lesn}%
so that the sequence $c_{l_k}\wedge x$ is
uniformly bounded and equicontinuous on $[0,t]$. Let $\tilde{c}$ be
its uniform limit on $[0,t]$. If $\tilde{c}^j(s)< x$ for all
$j=1,\dots,m$ we can argue as in the proof of the previous lemma to
see that $\tilde{c}=c$ on $[0,s]$. If $\tilde{c}^j(s)\geq x$ for any
$j=1,\dots,m$, we see that both $c^j$ and $\tilde{c}^j$ both reach
$x$ at the same point $s'\leq s$ and hence
$\tilde{c}^j(s)=c^j(s)\wedge x$. Hence $c^{j}_{l_k}\wedge x\to
c^j\wedge x$. 
Since $x$ is arbitrary, we see that $c_{l_k}^j\to c^j$
pointwise on $[0,t]$, even if $| c(t)|=\infty$.

Let $h_l=D_+c_l=(D_+c_l^1,\dots,D_+c_l^m)$ and $h=D_+ c=(D_+c^1,\dots,D_+c^m)$. 
We now prove that $h_l\to h$ in
the Skorohod $J_1$ topology if the explosion time $\tau$ is
infinite. 
Recall that $h=\sum_{i=1}^m f^{i }\circ c^{i}+g$ and that
when $\sigma_l=0$ then $h_l=\sum_{i=1}^m f_l^{i}\circ c_l^{i}+g_l$
while if $\sigma_l>0$ then $h_l^j(t)=[\sum_{i=1}^m f_l^{i,j}\circ
c_l^{i}(\floor{s/\sigma_l}\sigma_l)+g^j_l(\floor{s/\sigma_l}\sigma_l)]^+$. 
Assume that $\sigma_l=0$ for all $l$ (the arguments are analogous when $\sigma_l>0$), then the
assertion $h_l\to h$ is reduced in proving that : $f_l^i\circ
c^{i}\to f^i\circ c^{i}$ for all $i=1,\dots,m$, which is related to
the composition mapping on Skorohod space, and then deducing that
$\sum_{i=1}^m f_l^{i}\circ c_l^i +g_l \to\sum_{i=1}^m f^{i}\circ c^i
+g$, which is related to continuity of addition on Skorohod space.
Both continuity assertions require conditions to hold: the
convergence $f_l^i\circ c_l^{i}\to f^i\circ c^{i}$ can be deduced from \cite{MR2479479} 
if we prove that $f^i$ is continuous at every
point at which $(c^i)^{-1}$ is discontinuous, and the convergence
$\sum_{i=1}^m f_l^{i}\circ c_l^i +g_l \to\sum_{i=1}^m f^{i}\circ c^i
+g$ will hold because of \cite[Thm. 1.4]{MR561155} since we assume that
$f^i\circ c^i$ and $f^j\circ c^j$ will not jump at the same time nor as the same time as $g$. 
Hence the convergence $h_l\to h$ is reduced to proving
that $f^i$ is continuous at the discontinuities of $(c^i)^{-1}$. 
If $c^i$ is strictly increasing 
then $(c^i)^{-1}$ is continuous. 
When $c$ is not strictly increasing,
we will use the assumed uniqueness of \eqref{d1} to prove that $f^i$
is continuous at the discontinuities of $(c^i)^{-1}$. The proof
consists in two steps. 
First we will prove that $f^{ii}$ is
continuous at $c^{i}(s)$ and the we will use this fact to prove that
the rest of the components of $f^i$ is continuous at the same point.

We know prove that $f^{i,i}$ is continuous at the
discontinuities of $(c^i)^{-1}$. Suppose that $(c^i)^{-1}$ is
discontinuous at $x$. 
Let $s=(c^i)^{-1}(x-)$ and $t=(c^i)^{-1}(x)$. 
Then $c^i$ is constant on
on $[s,t]$ while $c^i<x$ on $[0,s)$ and $c^i> x$ on $(t,\infty)$. 
Since $D_+c^i=\sum_{i'=1}^m f^{i',i}\circ c^{i'} +g^i=0$ on $[s,t)$, we see
that $\sum_{i'=1,i'\neq i}^m f^{i',i}\circ c^{i'} +g^i$ is constant on
$[s,t)$. 
We assert that
\begin{lesn}
\inf\{y\geq0:f^{i,i}(y)=-\sum_{i'=1,i\neq i}^m f^{i',i}\circ c^{i'}(s)
+g^i(s)\}=x.
\end{lesn}Indeed, if $f^{i,i}$ reached $-\sum_{i'=1,i\neq i}^m f^{i',i}\circ c^{i'}(s)
+g^i(s)$ at $x'<x$, there would exist $s'<s$ such that
\begin{lesn}
f^{i,i}\circ c^{i}(s)+\sum_{i'=1,i\neq i}^m f^{i',i}\circ c^{i'}(s)
+g^i(s)=0\geq f^{i,i}\circ c^{i}(s')+\sum_{i'=1,i\neq i}^m f^{i',i}\circ
c^{i'}(s') +g^i(s')\geq 0.
\end{lesn}so that actually we have that $-\sum_{i'=1,i'\neq i}^m f^{i',i}\circ
c^{i'}(s) +g^i(s)$ is constant in $[s',t)$. Hence, $c^i$ has
spontaneous generation, which implies
that there are at least two solutions to \eqref{d1}: one that is constant on
$(s',s)$, and $c^i$. This contradicts the assumed uniqueness to
\eqref{d1}.

Having proved the continuity of $f^{i,i}$ at $x$, we need to prove that for each $j\not=i$, that $f^{i,j}$ is continuous at $x$. To this end let us recall that $c^i(s)=x$, and consider a sequence $\{x_n\}$ such that $x_n\uparrow x$, therefore there exists another sequence $\{s_n\}$ such that $s_n \uparrow s$. So using the continuity of $f^{i,i}$ at $x$ we have that
\begin{linenomath}
\begin{align*}
0
&\leq\lim_{n\to\infty} \left( f^{i,i}\circ c^i(s_n) +\sum_{j=1,j\not=i}^mf^{j,i}\circ c^j(s_n)+g^j(s_n)\right)
\\&\leq f^{i,i}\circ c^i(s) +\sum_{j=1,j\not=i}^mf^{j,i}\circ c^j(s)+g^j(s)=0.
\end{align*}\end{linenomath}Therefore we obtain that
\begin{lesn}
\lim_{n\to\infty} \sum_{j=1,j\not=i}^mf^{j,i}\circ c^j(s_n)+g^j(s_n)=\sum_{j=1,j\not=i}^mf^{j,i}\circ c^j(s)+g^j(s).
\end{lesn}
And hence we can conclude, using that the functions $f^{j,i}$ and $g^j$, are non-decreasing that 
\begin{lesn}
\lim_{n\to\infty}f^{j,i}\circ c^j(s_n)=f^{j,i} \circ c^j(s)\qquad\text{ for all $j\not=i$.}\qedhere
\end{lesn}
\end{proof}

\subsection{Applications of the stability analysis}
\label{applicationsOfStabilitySubsection}
In this subsection, we apply the stability analysis of Subsection \ref{stabilityAnalysisSubsection} to give a proof of Corollary \ref{GWCorollary}. 
\begin{proof}[Proof of Corollary \ref{GWCorollary}]
We first analyze the action of scaling on $C^{j,l}$. Since
\begin{align*}
\frac{1}{b^j_l}C^{j,l}_{a_l t}
&=\int_0^{a_l t} \bra{\frac{k^j_l}{b^j_l}+\frac{1}{b^j_l}\sum_{i=1}^m \imf{X^{i,j,l}}{C^{j,l}_{\floor{s}}} }\, ds
\\&=\int_0^t  \bra{\frac{k^j_l}{b^j_l}+\frac{a_l}{b^j_l}\sum_{i=1}^m \imf{X^{i,j,l}}{C^{j,l}_{a_l\, \floor{a_ls}/a_l}} }\, ds,
\end{align*}we see that $(C^{j, l}_{a_lt}/b^j_l, t\geq 0, 1\leq j\leq m)$ is the Euler type approximation of span $1/a_l$ applied to $(X^{i,j,l}(b^j_l t)\, b^j_l/a_l,t\geq 0, 1\leq i,j\leq m)$. 
Note that by hypothesis, the span $1/a_l$ goes to $0$ as $l\to\infty$. 
Also, the right-hand derivative of $C^{j,l}_{a_l t} /b^j_l$ equals $Z^{j,l}_{a_l t} a_l/b^j_l$.

Also, we have assumed $(X^{i,j,l}(b^j_l t)], b^j_l/a_l,t\geq 0, 1\leq j\leq m)$ converges to $X^{i,\cdot}$. 
Since $b^j_l/a^j_l\to 0$, we see that $X^{i,i}$ is spectrally positive. 
If $i\neq j$ then $X^{i,j}$ is a subordinator. 
(We have only assumed convergence in the Skorohod $J_1$ topology. However, the same arguments as in the proof of Corollary 7 of \cite{MR3098685} gives us convergence in the uniform $J_1$ topology in case of explosion.) 
Finally, since $X^{1,\cdot},\ldots, X^{m,\cdot}$ are independent, we are in position to construct the $\cb$ process $Z$ starting at $z^j$ and constructed from $X$ and $Y=0$ in Theorem \ref{affineProcessRepresentationTheorem}. 
From Theorem \ref{stochasticLimitTheorem}, we deduce that $(Z^{\cdot, l}_{a_l t} a_l/b^j_l)$ converges to $Z$ as $l\to\infty$, which proves Corollary \ref{GWCorollary}. 
\end{proof}

\bibliography{GenBib}
\bibliographystyle{amsalpha}
\end{document}